\documentclass[11pt]{article}
\usepackage[latin1]{inputenc}
\usepackage{amsmath}
\usepackage{amsfonts}
\usepackage{amssymb,amsthm, stmaryrd}
\usepackage{graphicx}
\usepackage{xcolor}
\usepackage{multicol}
\usepackage{soul}
\usepackage[normalem]{ulem}
\usepackage[displaymath, mathlines]{lineno}
\newcommand{\comment}[1]{}

\newcommand*\patchAmsMathEnvironmentForLineno[1]{%
  \expandafter\let\csname old#1\expandafter\endcsname\csname #1\endcsname
  \expandafter\let\csname oldend#1\expandafter\endcsname\csname end#1\endcsname
  \renewenvironment{#1}%
     {\linenomath\csname old#1\endcsname}%
     {\csname oldend#1\endcsname\endlinenomath}}%
\newcommand*\patchBothAmsMathEnvironmentsForLineno[1]{%
  \patchAmsMathEnvironmentForLineno{#1}%
  \patchAmsMathEnvironmentForLineno{#1*}}%
\AtBeginDocument{%
\patchBothAmsMathEnvironmentsForLineno{equation}%
\patchBothAmsMathEnvironmentsForLineno{align}%
\patchBothAmsMathEnvironmentsForLineno{flalign}%
\patchBothAmsMathEnvironmentsForLineno{alignat}%
\patchBothAmsMathEnvironmentsForLineno{gather}%
\patchBothAmsMathEnvironmentsForLineno{multline}%
}

\newtheorem{theorem}{Theorem}[section]

\newtheorem{problem}{Problem}

\newtheorem{corollary}[theorem]{Corollary}
\newtheorem{defi}[theorem]{Definition}
\newtheorem{lemma}[theorem]{Lemma}
\newtheorem{remark}[theorem]{Remark}

\newcommand{\Z}{\mbox{${\mathbb Z}$}}

\title{Unavoidable chromatic patterns in $2$-colorings\\ of the complete graph}

\begin{document}

\maketitle

\begin{center}

\begin{multicols}{2}

Yair Caro\\[1ex]
{\small Dept. of Mathematics\\
University of Haifa-Oranim\\
Tivon 36006, Israel\\
yacaro@kvgeva.org.il}

\columnbreak

Adriana Hansberg\\[1ex]
{\small Instituto de Matem\'aticas\\
UNAM Juriquilla\\
Quer\'etaro, Mexico\\
ahansberg@im.unam.mx}\\[2ex]

\end{multicols}

Amanda Montejano\\[1ex]
{\small UMDI, Facultad de Ciencias\\
UNAM Juriquilla\\
Quer\'etaro, Mexico\\
amandamontejano@ciencias.unam.mx}\\[4ex]

\end{center}

\begin{abstract}
Let $G$ be a graph with $e(G)$ edges. We say that $G$ is omnitonal if for every  sufficiently large $n$  there exists a minimum integer ${\rm ot}(n,G)$ such that the following holds true:
For any $2$-coloring  $f: E(K_n) \to \{red, blue \}$ with more than ${\rm ot}(n,G)$ edges from each color, and for any pair of non-negative integers $r$ and $b$ with $r+b = e(G)$, there is a copy of $G$ in $K_n$ with exactly $r$ red edges and $b$ blue edges.
We give a structural characterization of omnitonal graphs from which we deduce that omnitonal graphs are, in particular, bipartite graphs, and prove further that, for an omnitonal graph $G$, ${\rm ot}(n,G)  =  \mathcal{O}(n^{2 - \frac{1}{m}})$,  where $m = m(G)$  depends only on $G$.
We also present a class of graphs for which ${\rm ot}(n,G) = ex(n,G)$, the celebrated Tur\'an numbers. Many more results and problems of similar flavor are presented.
\end{abstract}

\section{Introduction and problem setting}

Our main interest in this paper is a certain kind of problems that lie in the junction of Ramsey theory, extremal graph theory, zero-sum Ramsey theory and interpolation theorems in graph theory; general references to these topics are  \cite{AlCa,BiDi90,Bol,Ca96, HaPl, Pun,Soi,Zhou}.  

We  consider $2$-colorings of the set of edges $E(K_n)$ of the complete graph $K_n$. 
Given a graph $G$ with $e(G)$ edges, non-negative integers $r$ and $b$ such that $r+b=e(G)$, and a $2$-coloring $f:E(K_n)\to \{ red , blue \}$, we say that $f$ induces an \emph{$(r,b)$-colored copy} of $G$, if there is a copy of $G$ in $K_n$ such that $f$ assigns the color red to exactly $r$ edges and the color blue to exactly $b$ edges of that copy of $G$. 

By Ramsey's theorem we know that any $2$-coloring of $E(K_n)$ (where $n$ is sufficiently large) induces either a $(e(G),0)$-colored copy of $G$, or a $(0,e(G))$-colored copy of $G$. To force the existence of $G$ with other color patterns, we need, as a natural minimum requirement, not only to ensure a large $n$, but also a minimum amount of edges of each color. In this paper, we  
study which graphs are unavoidable under a prescribed color pattern in every $2$-coloring of $E(K_n)$, whenever $n$ is sufficiently large and there are enough edges from each color. 
A similar approach has been studied in \cite{Alp, CuMo, FoSu}, where the emphasis is given on determining the minimum $n$ required to guarantee the existence of a given graph with a prescribed color pattern in every coloring of $K_n$ where each color appears in some positive fraction of the edges of $K_n$. In contrast, our approach has more like a Tur\'an flavor in the sense that we focus our attention, on the one hand, on the maximum edge number that can have the smallest color class in a $2$-coloring of $E(K_n)$ which is free of a copy of the given graph in the prescribed pattern, and, on the other, in characterizing the extremal colorings.

Observe that, in case $e(G) \equiv 0$ (mod $2$), the study of the existence of a zero-sum copy of $G$ over $\Z$-weightings of $E(K_n)$, in particular over $\{ -1 ,1\}$-weightings of $E(K_n)$, carries along similarity to classical Ramsey theory by simply defining all red edges to have weight $-1$ and all blue edges to have weight $1$. Thus, a zero-sum copy of $G$ translates into a copy of $G$ with equal number of red and blue edges, or equivalently  to an $(e(G)/2 , e(G)/2 )$-colored copy of $G$. The study of the existence of such a balanced copy will be one of the purposes of this paper, thus further developing the line of research studied in \cite{CaYu, CHM1, CHM2}. The second problem we will focus on, and in which we will make use of interpolation techniques, deals with the existence of an $(r ,b)$-colored copy of $G$ for every pair of non-negative integers $r$ and $b$ such that $r+b=e(G)$. As in the previous case, this problem is related to the study of the existence of a zero-sum copy of $G$ over $\Z$-weightings of $E(K_n)$ with range $\{ -p ,q\}$, where $p$ and $q$ are positive integers with $\gcd(p,q)=1$ and $e(G) \equiv 0$ (mod $p+q$). These problems will lead us to the definition of two graph families which will be the center of this work: balanceable graphs and omnitonal graphs.\\


\subsection{Notation}
For a given graph $G$, we use $V(G)$ and $E(G)$ to denote the sets of vertices and, respectively, of edges of $G$. Given a partition of the vertex set $V(G) = X \cup Y$, we denote by $E(X,Y)$ the set of edges of $G$ with one end in $X$ and the other one in $Y$. Also, we define $n(G) = |V(G)|$, $e(G) = |E(G)|$ and $e(X,Y) = |E(X,Y)|$. For a set $W\subseteq V(G)$, $G[W]$ stands for the subgraph of $G$ induced by the vertices in $W$. Similarly, if $F \subseteq E(G)$, $G[F]$ denotes the graph induced by the edges from $F$. A set $W\subseteq V(G)$ is called \emph{independent} if $G[W]$ is edgeless.

Let $p$, $q$ and $k$ be non-negative integers. We will denote with $K_{p,q}$ the \emph{complete bipartite graph} with one partition set having $p$ vertices an the other $q$. The graph  $K_{1,k}$ will be also called a \emph{$k$-star}. Moreover, a \emph{$k$-path $P_k$} denotes a path on $k+1$ vertices and $k$ edges.  A graph $G$ is a  \emph{$(p,q)$-split graph} if there is a partition of the vertex set $V(G) = X \cup Y$ with $|X|=p$ and  $|V|=q$ such that $G[X] \cong K_p$  induces a complete graph and $Y$ is an independent set. Furthemore, the split graph $G$  is called \emph{complete} if $G[E(X,Y)] \cong K_{p,q}$. 

A \emph{coloring} of the edges of a graph $G$ is a mapping $f: E(G) \to \mathcal{C}$ to a set of colors $\mathcal{C}$. If $|\mathcal{C}| = c$, we talk about a  \emph{$c$-edge coloring}, or for short a \emph{$c$-coloring}. In the whole paper, we will deal with red and blue colorings on the complete graph, that is, we will consider mappings $f: E(K_n) \to  \{red, blue\}$. Since such a coloring induces a partition of the edge set of $K_n$ into the set of red edges and the set of blue edges, we can talk about the \emph{red graph} $R$ and the \emph{blue graph} $B$ induced by the red and, respectively, blue edges of this coloring. In the following, in order to avoid talking every time of the mapping, we will talk about the $2$-coloring $E(K_n) = E(R) \cup E(B)$, assuming implicitly that $R$ and $B$ are the graphs induced by the red and the blue edges.

\subsection{Problem setting}

\noindent
{\bf Balanceable graphs}\\
For a given graph $G$, we say that a $2$-coloring $E(K_n) = E(R) \cup E(B)$ contains a \emph{balanced} copy of $G$, if,  in the so colored $K_n$, we can find an $(e(G)/2, e(G)/2)$-colored copy of $G$ in case $e(G) \equiv 0$ (mod $2$), and a $(\lceil e(G)/2\rceil, \lfloor e(G)/2\rfloor)$-colored copy of $G$ or an $(\lfloor e(G)/2\rfloor,\lceil e(G)/2\rceil)$-colored copy of $G$ in case $e(G)\equiv  1$ (mod $2$).

\begin{defi}\label{def:bal}
For a given graph $G$ let ${\rm bal}(n,G)$ be the minimum integer, if it exists, such that any $2$-coloring $E(K_n) = E(R) \cup E(B)$ with $\min \{e(R), e(B)\} > {\rm bal}(n,G)$ contains a balanced copy of $G$. If ${\rm bal}(n,G)$ exists for every sufficiently large $n$, we say that $G$ is \emph{balanceable}\footnote{\label{foot:str_bal}For graphs with an odd number of edges, there is a stronger notion  of balanceable graphs that one can naturally consider: instead of seeking for one of both, a $(\lceil e(G)/2\rceil, \lfloor e(G)/2\rfloor)$-colored copy or a $( \lfloor e(G)/2\rfloor,\lceil e(G)/2\rceil)$-colored copy of $G$, one can study the existence  of both  patterns.}. For a balanceable graph $G$, let ${\rm Bal}(n,G)$ be the family of graphs with exactly ${\rm bal}(n,G)$ edges such that a coloring $E(K_n) = E(R) \cup E(B)$ with $\min \{e(R), e(B)\} = {\rm bal}(n,G)$ contains no balanced copy of $G$ if and only if $R$ or $B$ is isomorphic to some $H \in {\rm Bal}(n,G)$.
\end{defi}

We shall be interested in finding balanceable graphs as well as in, if possible, determining ${\rm bal}(n,G)$ or, if not, in finding good estimates. If ${\rm bal}(n,G)$ is known, we also consider the problem of characterizing the extremal colorings, meaning that we aim to determine the family ${\rm Bal}(n,G)$.
Note that, to prove that a graph $G$ is not balanceable, it is enough to exhibit infinitely many values of $n$ for which there is a $2$-edge coloring of $K_n$ with  the same or almost the same (differing by at most one unit) number of red and blue edges without a balanced copy of $G$.

As we shall show, there is a plethora of balanceable graphs, including $K_4$  but not any other complete graph $K_m$ with $e(K_m)$ even, as proved in \cite{CHM1}, where ${\rm bal}(n,K_4)$ and ${\rm Bal}(n,K_4)$ were  also determined.

The connection to the zero-sum analogue with $\{ -1 ,1\}$-weightings is already explained in the introduction section. The case when $e(G) \equiv 1$ (mod $2$) has no direct analogue as a zero-sum problem. However, we mention that such odd-case variations have been considered in \cite{CHM2} in the context of $\{ -1 ,1\}$-sequences. This  establishes the bridge between the current paper and the results on $\{ -1 ,1\}$-weightings given for instance in \cite{CaYu, CHM1, CHM2}.\\

\noindent
{\bf Omnitonal graphs}\\
As we will formally define below, omnitonal graphs will be those graphs that appear in all possible tonal variations of red and blue in every $2$-edge coloring of the complete graph, as long as the latter is large enough.

\begin{defi}
For a given graph $G$, let ${\rm ot}(n,G)$ be the minimum integer, if it exists, such that any $2$-coloring $E(K_n) = E(R) \cup E(B)$ with $\min \{e(R), e(B)\} > {\rm ot}(n,G)$ contains an $(r, b)$-colored copy of $G$ for any $r\geq 0$ and $b\geq 0$ such that $r+b=e(G)$. If ${\rm ot}(n,G)$ exists for every sufficiently large $n$, we say that $G$ is \emph{omnitonal}. For an omnitonal graph $G$, let ${\rm Ot}(n,G)$ be the family of graphs with exactly ${\rm ot}(n,G)$ edges such that a coloring $E(K_n) = E(R) \cup E(B)$ with $\min \{e(R), e(B)\} = {\rm ot}(n,G)$ contains no $(r,b)$-colored copy of $G$ for some pair $r,b \ge 0$ with $r+b = e(G)$ if and only if $R$ or $B$ is isomorphic to some $H \in {\rm Ot}(n,G)$.
\end{defi}

We shall be interested in finding omnitonal graphs as well as in, if possible, determining ${\rm ot}(n,G)$ or, if not, finding good estimates. If we are able to determine ${\rm ot}(n,G)$, we also consider the problem of characterizing the extremal colorings, that is, in finding ${\rm Ot}(n,G)$.

To determine if a graph $G$ is omnitonal is a problem within the scope of interpolation theorems in graph theory; as we shall see, we will incorporate proof techniques typical to this area together with techniques typical in Ramsey theory and extremal graph theory to obtain results concerning omnitonal graphs.

As in the case of balanceable graphs, to prove that a graph $G$ is not omnitonal, it is enough to exhibit infinitely many values of $n$ for which there is a $2$-edge coloring of $K_n$ with  the same or almost the same (differing by at most one unit) number of red and blue edges without an $(r, b)$-colored copy of $G$ for some $r\geq 0$ and $b\geq 0$ such that $r+b=e(G)$.

Observe that, if a graph $G$  is omnitonal, then it is balanceable  (by choosing $r$ and $b$ such that $r=b$ in case $e(G)\equiv  0$ (mod $2$) and $|r - b| =1$ in case $e(G)\equiv  1$ (mod $2$)) but not necessarily vice versa. For instance, $K_4$ is balanceable (see Theorem \ref{thm:Km}) but the following construction shows that $K_4$ is not omnitonal: Consider a partition of the vertex set $V(K_n) = A \cup B$ and a coloring of the edges such that all edges inside $A$ are colored red, all other edges are colored blue and we choose $|A|$ and $|B|$ so that $e(R)$ and $e(B)$ are equal (this can be done for infinitely many values of $n$, see Lemma \ref{la:infty_many_n}). Evidently, there are no $(2,4)$-colored copies of $K_4$ in this coloring. Moreover, concerning $K_m$ for arbitrary $m \ge 3$, observe that the pattern $(r,b) = (2, {m \choose 2} -2)$ can never be realized for a copy of $K_m$ in the coloring of $K_n$ given above, implying that no complete graph is omnitonal.

The connection to the zero-sum problem with $\{-p,q\}$-weightings is explained below.

\begin{remark}
Let $p$ and $q$ be positive integers with $\gcd(p,q)=1$ and let $G$ be an omnitonal  graph with $e(G) \equiv 0$ {\rm(mod  $p+q$)}. Then, for large enough $n$, any coloring $f:E(K_n) \rightarrow \{-p, q\}$ with $\min \{|f^{-1}(-p)|,|f^{-1}(q)| \} > {\rm ot}(n,G)$ contains a zero-sum copy $G^*$ of $G$ (that is, a copy $G^*$ of $G$ where $\sum_{e\in E(G^*)}f(e)=0$). To see this, define another coloring $g:E(K_n) \rightarrow \{red, blue\}$ where $g(e)=red$ iff $f(e)=-p$ and $g(e)=blue$ iff $f(e)=q$. Since $G$ is omnitonal, and $g$ is such that  $\min \{e(R),e(B) \} > {\rm ot}(n,G)$, we can find an $(r,b)$-colored copy $G^*$ of $G$ with $r=\frac{p e(G)}{p+q}$ and $b=\frac{q e(G)}{p+q}$. Then $G^*$ is a zero-sum copy of $G$ under coloring $f$: 
\[\sum_{e\in E(G^*)}f(e)=-qr+pb=\frac{-qp e(G^*)}{p+q}+\frac{pqe(G^*)}{p+q}=\frac{(-qp+qp)e(G^*)}{p+q}=0.\]
\end{remark}

We are unaware of a systematic study along the lines suggested by the omnitonal graphs, which we start here.

\subsection{Results}\label{sec:results}

After this introductory part, the article is divided into four sections. In Section \ref{sec:st}, we establish structural results for both balanceable and omnitonal graphs. In Corollary \ref{cor:char-BP} and Theorem~\ref{thm:fsp}, we give necessary and sufficient conditions for a graph to be balanceable and, respectively, omnitonal. The characterization of balanceable graphs is actually a particular case of Theorem \ref{thm:char-r-tonal} which provides necessary and sufficient conditions for a graph to be $r$-tonal (see Definition \ref{def:rtonal}). From these characterizations we derive easily, for example, that  trees are omnitonal and, therefore, also balanceable graphs. 
All these results are consequences of Theorem \ref{thm:RTZ}, proved in Section \ref{sec:RTZ}, which is a new version of a result by Cutler and Mont\'agh \cite{CuMo} solving a conjecture proposed by Bollob\'as (see \cite{CuMo}).  

When considering the problem of determining if a graph is balanceable or omnitonal, the study of two particular $2$-colorings of $E(K_{2t})$, which we will call \emph{type-$A$} (where one color forms a clique of order $t$) and \emph{type-$B$} (where one color forms two disjoint  cliques of order $t$) colorings, arises naturally.  It was shown in \cite{CuMo} that, for sufficiently large $n$, every $2$-edge-coloring of $K_n$ with a positive fraction of edges of each color contains a copy of a type-$A$ or a type-$B$ colored $K_{2t}$. Without seeking a sharp bound on $n$, our result prescinds from the quadratic amount of edges of each color, and replaces it with a sub-quadratic constraint, implying that, in case of existence, ${\rm bal}(n,G)$ and ${\rm ot}(n,G)$ are always sub-quadratic as functions of $n$. The nature of our proof of  Theorem \ref{thm:RTZ} avoids probabilistic arguments and relies on the Ramsey Theorem, the Tur\'an numbers and the Zarankiewicz numbers. It also prevents to get good upper bounds on ${\rm bal}(n,G)$ and ${\rm ot}(n,G)$, which are left for further research to be improved. 

Another major idea in this work, presented  in Section \ref{sec:AMO}, is the study of a class of graphs called \emph{amoebas} (see Definition \ref{def:amo}).  These graphs were  developed here along the proof techniques used in interpolation theorems in graph theory, building upon ideas from \cite{CHM1} and \cite{CHM2}. In particular, we prove that every amoeba is balanceable (Theorem \ref{teo:amoebaBal}) and that every bipartite amoeba $G$ is omnitonal with ${\rm ot}(n,G)={\rm ex}(n,G)$ and ${\rm Ot}(n,G)={\rm Ex}(n,G)$, where ${\rm ex}(n,G)$ and ${\rm Ex}(n,G)$ are the Tur\'an number for $G$ and the corresponding family of extremal graphs (Theorem \ref{thm:amoebasOT}).  
 
In Section \ref{sec:BAL}, we determine ${\rm bal}(n,G)$ as well as ${\rm Bal}(n,G)$ for paths and stars with an even edge number. Further, in Section \ref{sec:OMNI}, we determine ${\rm ot}(n,G)$ and ${\rm Ot}(n,G)$ for stars. Moreover, we show that ${\rm ot}(n,T) \le (k-1)n$ for every tree $T$ on $k$ edges. Since paths are bipartite amoebas, ${\rm ot}(n,P_k)$ and ${\rm Ot}(n,P_k)$ are already covered in Section \ref{sec:AMO}.

Finally, in Section~\ref{sec:OP}, we discuss further variants of these concepts and present several open problems.

\section{Structural results}\label{sec:st}

\subsection{Characterization of balanceable and omnitonal graphs}\label{sec:RTZ}

Let $t$ and $n$ be integers with $1 \le t < n$. A $2$-edge-colored complete graph $K_n$ is said to be of  \emph{type $A(t)$} if the edges of one of the colors induce a complete graph $K_t$, and it is of \emph{type $B(t)$} if the edges of one of the colors induce a complete bipartite graph $K_{t,n-t}$. If $n = 2t$, we eliminate the parameter $t$ and write for short \emph{type-$A$} and  \emph{type-$B$} colorings.

For a given graph $G$, we denote by $R(G,G)$ the \emph{$2$-color Ramsey number}, that is, the minimum integer $R(G,G)$ such that, whenever $n\geq R(G,G)$, any coloring $E(K_n) = E(R) \cup E(B)$ contains either a blue or a red copy of $G$. 
For a given graph $G$, we denote by ${\rm ex}(n,G)$ the \emph{Tur\'an number for $G$}, that is, the maximum number of edges in a graph with $n$ vertices containing no copy of $G$. The well-known K\H{o}vari-S\'os-Tur\'an theorem \cite{KST} implies that, for the balanced complete bipartite graph $K_{t,t}$,
\begin{equation}\label{eq:t}
{\rm ex}(n,K_{t,t}) < \frac{1}{2}\left((t-1)^{1/t} n^{2-1/t} + \frac{1}{2}(t-1)n\right).
\end{equation}

For a given positive integer $t$, we denote by $z(n,t)$ the \emph{Zarankiewicz number}, that is, the maximum number of edges in a bipartite graph with $n$ vertices in each part, containing no copy of $K_{t,t}$. Here again, the K\H{o}vari-S\'os-Tur\'an theorem yields the following upper bound for $z(n,t)$:
\begin{equation}\label{eq:z}
z(n,t) < (t-1)^{1/t} n^{2-1/t} + \frac{1}{2}(t-1)n.
\end{equation}

The following Theorem \ref{thm:RTZ} is a new version of a result first proved by Cutler and Mont\'agh \cite{CuMo} solving a conjecture raised by Bollob\'as (see \cite{CuMo}) about the existence, for $n$ sufficiently large, of a type-$A$ or type-$B$ colored $K_{2t}$ in every $2$-edge-coloring of $K_n$ with a positive fraction of edges of each color. The bound on the Ramsey parameter concerning this problem was further improved by  Fox and Sudakov in \cite{FoSu}. In both papers, the authors explicitly assume $\min\{e(R),  e(B)\} = \epsilon {n \choose 2}$ for some $\epsilon >  0$  and estimate an upper bound on the smallest $n$ for which this $\epsilon$-balancing forces a type-$A$ or a type-$B$ colored copy of $K_{2t}$. Not seeking a sharp bound on $n$, our result, in contrast, prescinds from the $\epsilon$-balancing restriction and replaces it with a  weaker constraint, which in turn allows us to give a subquadratic bound, as a function of $n$, on the minimum number of edges of each color required on the $2$-edge coloring of the $K_n$. The proof of our version avoids probabilistic arguments and uses only the classical Ramsey and Tur\'an numbers for complete bipartite graphs and the Zarankiewicz numbers instead.

\begin{theorem}\label{thm:RTZ}
Let $t$ be a positive integer. For all sufficiently large $n$, there exists a positive integer $m = m(t)$ and a number $\varphi(n,t) = \mathcal{O}(n^{2-\frac{1}{m}})$ such that any coloring $E(K_n) = E(R) \cup E(B)$ with $\min \{ e(R) ,e(B) \} \geq \varphi(n,t)$ contains a type-$A$ or a type-$B$ colored copy of $K_{2t}$.
\end{theorem}

\begin{proof}
Let $q\geq t$ be an integer such that 
\begin{equation}\label{eq:q}
(t-1)^{1/t}(2q)^{2-1/t}+(t-1)q+1\leq 2q^2,
\end{equation}
and set $m=R(K_q,K_q$). Now define 
\[\varphi(n,t) = {\rm ex}(n, K_{m,m}) + m(m-1) + 2m(n-2m)+1,\]
which, by (\ref{eq:t}), is clearly $\mathcal{O}(n^{2-\frac{1}{m}})$. Assume $n$ to be large enough such that we can take a coloring $E(K_n) = E(R) \cup E(B)$ with $\min \{ e(R) ,e(B) \} \geq \varphi(n,t)$. This is possible since $\varphi(n,t) = o(n^2)$. By definition, there is a monochromatic, say red, copy of $K_{m,m}$ in $K_n$. Let $X\cup Y$ be a vertex set partition of such a red copy of $K_{m,m}$, where $|X|=|Y|=m$ and all edges between $X$ and $Y$ are red. Consider now the complete graph $K_{n-2m}$ obtained from $K_n$ by removing the vertex set $X\cup Y$. Since we lose at most 2${m\choose 2}+2m(n-2m)$ blue edges, by the definition of $\varphi(n,t)$, there are at least ${\rm ex}(n-2m,K_{m,m})+1$ blue edges in $K_{n-2m}$. Hence, 
there is a blue copy of $K_{m,m}$ in  $K_{n-2m}$. Let $Z\cup W$ be a vertex set partition of such a blue copy of $K_{m,m}$, where $|Z|=|W|=m$ and all edges between $Z$ and $W$ are blue. Observe that the red and the blue copies of $K_{m,m}$ that we obtain are vertex disjoint.

Now consider the $2$-edge colored graph induced by $X\cup Y$. By the definition of $m$, we know that there are monochromatic copies of $K_q$ inside both $X$ and $Y$. If at least one of these monochromatic copies of $K_q$ is blue then, since all edges between $X$ and $Y$ are red, we will have a copy of $K_{2q}$ which is either of type $A$ or of type $B$; and since $q\geq t$ we are done in this case. Otherwise, we get two red monochromatic copies of $K_q$, one inside $X$ an the other inside $Y$, which indeed is a monochromatic, red, copy of $K_{2q}$. Similarly, by looking at the $2$-edge colored graph induced by $Z\cup W$, either we are done or we get a blue copy of $K_{2q}$. Hence, we can assume that we have two vertex disjoint monochromatic copies of $K_{2q}$, one red and one blue. Call $C$ the set of vertices of the red one, and $D$ the set of vertices of the blue one.

Finally, we consider the $2$-edge colored complete bipartite graph, $K_{2q,2q}$, induced by $C\cup D$. Clearly, one of the colors, say red, has at least half of the edges. In other words, there are at least $\frac{1}{2}(2q)^2=2q^2$ red edges in $K_{2q,2q}$. By computing the upper bound (\ref{eq:z}) of the Zarankiewicz number $z(2q,t)$, we obtain the left hand of (\ref{eq:q}). Thus, by the definition of $q$, we gain a monochromatic copy of $K_{t,t}$ in $K_{2q,2q}$. That is, there are subsets $C'\subset C$ and $D'\subset D$, with $|C'|=|D'|=t$,  such that all edges between $C'$ and $D'$ are red. Observe that the $2$-edge colored  complete graph $K_{2t}$ induced by $C' \cup D'$ is of type $A$, which completes the proof.
\end{proof}

As we said before, for a given graph $G$, Ramsey's Theorem guaranties, for large enough $n$, the existence of either a $(0, e(G))$-colored copy or a $(e(G), 0)$-colored copy of $G$ in every $2$-coloring of $E(K_n)$, while, to force the existence of other color patterns, there also have to be enough edges from each color. The precise amount of edges from each color needed to this aim, if it exists, is the parameter that we are interested in. The following definition extends Definition  \ref{def:bal} from a balanced proportion of the colors to other proportion variations.

\begin{defi}\label{def:rtonal}
Let $G$ be a graph and $r$ an integer with $0 < r \le \left\lfloor e(G)/2\right\rfloor$. 
Let ${\rm bal_r}(n,G)$ be  the minimum integer, if it exists, such that every $2$-coloring $E(K_n) = E(R) \cup E(B)$ with $\min \{e(R), e(B)\} > {\rm bal_r}(n,G)$ contains, either an $(r, e(G)-r)$-colored copy of $G$, or an $(e(G)-r, r)$-colored copy of $G$. If ${\rm bal_r}(n,G)$ exists for every $n$ sufficiently large, we say that $G$ is \emph{$r$-tonal}.
\end{defi}

Observe that, if ${\rm bal_r}(n,G)$ exists, then ${\rm bal_r}(n,G) \le \frac{1}{2} \binom{n}{2}$. Clearly this happens, too, for omnitonal graphs: ${\rm ot}(n,G) \le \frac{1}{2} \binom{n}{2}$.\\

Theorem \ref{thm:RTZ} allows us to give a characterization of $r$-tonal (thus also balanceable) graphs and omnitonal graphs. We will use the following lemma that follows directly from Lemmas 3.1 and 3.2 given in \cite{CHM1}.

\begin{lemma}[\cite{CHM1}]\label{la:infty_many_n}
For infinitely many positive integers $n$, we can choose $t=t(n)$ in a way that 
the type-$A(t)$ coloring of $K_n$ is balanced.  Also,  for infinitely many positive integers $n$, we can choose $t=t(n)$ in a way that  the type-$B(t)$ coloring of $K_n$ is balanced.
\end{lemma}

\begin{theorem}\label{thm:char-r-tonal}
Let $G$ be a graph and let $r$ be an integer with $0 < r \le \left\lfloor e(G)/2\right\rfloor$. Then $G$ is $r$-tonal if and only if $G$ has both a partition $V(G)=X \cup Y$ and a set of vertices $W\subseteq V(G)$ such that $e(X,Y), e(G[W]) \in \{r,  e(G) - r\}$.
\end{theorem}

\begin{proof}
Suppose that $G$ is $r$-tonal. Let $n$ be large enough such that ${\rm bal_r}(n,G)$ exists and chosen such that there is a balanced type-$A(t)$ coloring of $K_n$ for some $t = t(n)$, which is possible by Lemma \ref{la:infty_many_n}. Suppose, without loss of generality, that the graph induced by the red edges in such a coloring of $K_n$ is isomorphic to $K_t$. Since $G$ is $r$-tonal and ${\rm bal_r}(n,G) \le \frac{1}{2} \binom{n}{2} = e(R) = e(B)$, there must be a copy of $G$ in $K_n$ with $r$ or $e(G) - r$ red edges, implying that there is a set $W \subseteq V(G)$ with $e(G[W]) \in \{r, e(G) - r\}$.
Analogously, we take now an $n$ large enough such that ${\rm bal_r}(n,G)$ exists and chosen such that there is a balanced type-$B(t)$ coloring of $K_n$ for some $t = t(n)$, which, again, is possible by Lemma \ref{la:infty_many_n}. Suppose, without loss of generality, that the graph induced by the red edges in such a coloring of $K_n$ is isomorphic to $K_{t,n-t}$. Since $G$ is $r$-tonal and ${\rm bal_r}(n,G) \le \frac{1}{2} \binom{n}{2} = e(R) = e(B)$, there must be a copy of $G$ in $K_n$ with $r$ or $e(G) - r$ edges, implying that there is a partition $V(G)=X \cup Y$  with $e(X,Y) \in \{r, e(G) - r\}$. 

Conversely, suppose that $G$ has both a partition $V(G)=X \cup Y$ and a set of vertices $W\subseteq V(G)$ such that $e(X,Y), e(G[W]) \in \{r, e(G) - r\}$.  Let $E(K_n) = E(R) \cup E(B)$ be an edge coloring of $K_n$  with $\min \{ e(R) ,e(B) \} \geq \varphi(n,t)$, where $t = n(G)$ and $\varphi(n,t)$ is like in Theorem \ref{thm:RTZ}. Hence, for $n$ sufficiently large, there is a type-$A$ or a type-$B$ copy of $K_{2t}$. If this copy is of type $A$, say we have one red $K_t$ and one blue $K_t$ and all edges in between are blue, then we can find a copy of $G$ placing the set $W$ inside the red $K_t$ and the other vertices inside the blue $K_t$.
If this copy is of type $B$, say we have two red $K_t$'s joined by blue edges, then we can find a copy of $G$ placing the edge cut $e(X,Y)$ such that $X$ and $Y$ are each in one of the red $K_t$'s. \end{proof}

Since a graph $G$ is balanceable if and only if it is $\left\lfloor e(G)/2\right\rfloor$-tonal, the following corollary is immediate from Theorem \ref{thm:char-r-tonal}.

\begin{corollary}\label{cor:char-BP}
A graph $G$ is balanceable if and only if $G$ has both a partition $V(G)=X \cup Y$ and a set of vertices $W\subseteq V(G)$ such that $ e(X,Y), e(G[W]) \in \{\lfloor \frac{1}{2}e(G) \rfloor, \lceil \frac{1}{2}e(G) \rceil\}$.
\end{corollary}

Adopting the same proof technieque from Theorem \ref{thm:char-r-tonal}, we obtain the next result.

\begin{theorem}\label{thm:fsp}
A graph $G$ is omnitonal if and only if, for every integer $r$ with $0 \le r \le e(G)$, $G$ has both a partition $V(G)=X \cup Y$  and a set of vertices $W\subseteq V(G)$ such that $e(X,Y) = e(G[W])= r$.
\end{theorem}

\begin{proof}
Suppose that $G$ is omnitonal. Let $n$ be large enough such that ${\rm ot}(n,G)$ exists and chosen such that there is a balanced type-$A(t)$ coloring of $K_n$ for some $t = t(n)$, which is possible by Lemma \ref{la:infty_many_n}. Suppose, without loss of generality, that the graph induced by the red edges in such a coloring of $K_n$ is isomorphic to $K_t$. Since $G$ is omnitonal and ${\rm ot}(n,G) \le \frac{1}{2} \binom{n}{2} = e(R) = e(B)$, there must be a copy of $G$ in $K_n$ with $r$ red edges for every $0 \le r \le e(G)$.  This implies that there is a set $W \subseteq V(G)$ with $e(G[W]) = r$ for every $0 \le r \le e(G)$. 
Analogously, we take now an $n$ large enough such that ${\rm ot}(n,G)$ exists and chosen such that there is a balanced type-$B(t)$ coloring of $K_n$ for some $t = t(n)$, which, again, is possible by Lemma \ref{la:infty_many_n}. Suppose, without loss of generality, that the graph induced by the red edges in such a coloring of $K_n$ is isomorphic to $K_{t,n-t}$. Since $G$ is omnitonal and ${\rm ot}(n,G) \le \frac{1}{2} \binom{n}{2} = e(R) = e(B)$, there must be a copy of $G$ in $K_n$ with $r$ edges for every $0 \le r \le e(G)$. It follows that there is a partition $V(G)=X \cup Y$  with $e(X,Y) = r$ for every $0 \le r \le e(G)$. 

Conversely, suppose that $G$ has both a partition $V(G)=X \cup Y$  and a set of vertices $W\subseteq V(G)$ such that $e(X,Y) = e(G[W])= r$ for every $0 \le r \le e(G)$. Let $E(K_n) = E(R) \cup E(B)$ be an edge coloring of $K_n$  with $\min \{ e(R) ,e(B) \} \geq \varphi(n,t)$, where $t = n(G)$ and $\varphi(n,t)$ is like in Theorem \ref{thm:RTZ}. Hence, for $n$ sufficiently large, there is a type-$A$ or a type-$B$ colored copy of $K_{2t}$. If this copy is of type $A$, then there are two possibilities: either we have one red $K_t$ and one blue $K_t$ and all edges in between are blue or the colors are reversed. In the first case, we can use a set $W$ with $e(G[W])= r$ to find a copy of $G$ with $r$ red edges and $e(G)-r$ blue edges. In the second, we can use a set $W'$ with $e(G[W'])= e(G)-r$ to find a copy of $G$ with $r$ red edges and $e(G)-r$ blue edges. The case of having a type-$B$ colored copy of $K_{2t}$ is similar.
\end{proof}

Having determined the structure of $r$-tonal and omnitonal graphs in Theorems \ref{thm:char-r-tonal} and \ref{thm:fsp}, we learn from Theorem \ref{thm:RTZ} that ${\rm bal_r}(n,G)$, thus also ${\rm bal}(n,G)$, and ${\rm ot}(n,G)$ are all sub-quadratic as functions of $n$.

\begin{corollary}\label{cor:subq}
If $G$ is an $r$-tonal graph, then ${\rm bal_r}(n,G) = \mathcal{O}(n^{2-\frac{1}{m}})$, where $m = m(G)$ depends only on $G$ (this holds true, in particular, for balanceable graphs). Also, if $G$ is omnitonal, then ${\rm ot}(n,G) = \mathcal{O}(n^{2-\frac{1}{m}})$, where $m = m(G)$ depends only on $G$.
\end{corollary}

The next result concerns the chromatic number of omnitonal graphs.

\begin{theorem}
Omnitonal graphs are bipartite.
\end{theorem}
\begin{proof}
Take an $n$ large enough such that ${\rm ot}(n,G)$ exists and chosen such that there is a balanced type-$B(t)$ coloring of $K_n$ for some $t = t(n)$, which is allowed by Lemma \ref{la:infty_many_n}. Suppose, without loss of generality, that the graph $R$, induced by the red edges in such a coloring of $K_n$, is isomorphic to $K_{t,n-t}$. Since $G$ is omnitonal and ${\rm ot}(n,G) \le \frac{1}{2} \binom{n}{2} = e(R) = e(B)$, there must be a red copy of $G$ contained in the red graph $R$, which is bipartite. Hence, $G$ must be bipartite.
\end{proof}

\begin{remark}
There are bipartite graphs which are not balanceable and hence neither omnitonal. For example, cycles $C_t$ of lenght $t \equiv 2 \;({\rm mod} \; 4)$ do not appear balanced in any type-$B$ colored $K_n$. Moreover, balanceable graphs are not necessarily bipartite: $K_4$ is an example (see Theorem \ref{thm:Km}). 
\end{remark}

In contrast with the fact that certain even length cycles are not balanceable as given in the above remark, the situation for trees (and for forests) is completely different.

\begin{theorem}\label{thm:treeOT}
Every tree is omnitonal.
\end{theorem}

\begin{proof}
Let $T$ be a tree. According to Theorem \ref{thm:fsp}, we have to verify that, for every integer $r$ with $0 \le r \le e(T)$, $T$ has both a partition $V(T)=X \cup Y$ and a set of vertices $W\subseteq V(T)$ such that $e(X,Y) = e(G[W])= r$. We proceed by induction on $e(T)$. If $e(T)=1$ then both conditions are clearly satisfied for $0\leq r\leq 1$. Let $T$ be a tree with $e(T)=m$, and let $v\in V(T)$ be a leaf where $u$ is the only vertex of $T$ adjacent to $v$.  By the induction hypothesis, the tree $T'=T- \{v\}$ satisfies that, for every $0\leq r\leq m-1=e(T')$, there are both a partition $V(T')=X' \cup Y'$ and a set of vertices $W\subseteq V(T')$ such that $e(X',Y') =e(T'[W])= r$.  Note that for every $0\leq r\leq m-1$ the subset   $W\subseteq V(T')\subset V(T)$ satisfies $e(T[W]) =r$. Likewise, for every $0\leq r\leq m-1$ we can obtain a partition $V(T)=X\cup Y$ with $e(X,Y) =r$ by taking $X=X'\cup \{v\}$ and $Y=Y'$ if $u\in  X'$, or $X=X'$ and $Y=Y'\cup \{v\}$ if $u\in Y'$. To show that there are both a partition $V(T)=X \cup Y$ and a set of vertices $W\subseteq V(T)$ such that $e(X,Y) =e(T[W])=m=e(T)$ is trivial and the proof is concluded. 
\end{proof}

It is not difficult to see that the disjoint union of two  omnitonal  graphs is again an omnitonal  graph. Hence,  it follows directly from Theorem \ref{thm:treeOT} that every forest is omnitonal.


\subsection{Amoebas}\label{sec:AMO}

In this section, we describe a class of graphs which we call amoebas. We are interested in such graphs since, as we shall see below, amoebas are balanceable and provide a wide family of omnitonal graphs, too. 

Given a graph $G$ of order $n(G)$ embedded in a complete graph $K_n$, where  $n\geq n(G)$, we say that $H$ (also embedded in $K_n$) is obtained from $G$ by an \emph{edge-replacement}, if for some $e_1\in E(G)$ and $e_2\in E(K_n)\setminus E(G)$, $E(H)= (E(G) \setminus \{e_1\}) \cup \{e_2\}$. Isolated vertices will play no role here, so all graphs  considered further on may be the ones induced by its corresponding edge set.

\begin{defi}\label{def:amo}
A graph $G$ is an \emph{amoeba} if there exist $n_0=n_0(G) >n(G)$, such that for all $n\geq n_0$ and any two copies $F$ and $H$ of $G$ in $K_n$, there is a chain $F = G_0, G_1, G_2, \cdots, G_t = H$ such that, for every $i \in \{1,\cdots, t\}$,  $G_i \cong G$ and $G_i$ is obtained from $G_{i-1}$ by an edge-replacement. 
\end{defi}

For example, it is not hard to see that a path $P_k$ is an amoeba for every $k\geq 1$, while a cycle $C_k$ is not an amoeba for any $k \geq 3$.\\

The following is a basic interpolation lemma for amoebas.

\begin{lemma}\label{lem:basic}
Let $G$ be an amoeba and consider a $2$-coloring $E(K_n) = E(R) \cup E(B)$ where  $n\geq n_0(G)$. Let $\alpha,\beta,\alpha', \beta'$ be integers such that $\alpha+\beta= \alpha'+\beta' = e(G)$ and $0\leq \alpha \leq \alpha'$ and $0\leq \beta' \leq \beta$. If there are both, an $(\alpha,\beta)$- and an $(\alpha',\beta')$-colored copy of $G$, then, there is an $(r,b)$-colored copy of $G$ for all integers $r$ and $b$ such that $r+b=e(G)$, $\alpha \leq r\leq  \alpha'$ and $\beta' \leq b\leq \beta$. 
\end{lemma}

\begin{proof}
Under the hypothesis of the lemma, let $F$ be an $(\alpha, \beta)$-colored copy of  $G$, and $H$ be an $(\alpha',\beta')$-colored copy of $G$ with $0\leq \alpha \leq \alpha'$ and $0\leq \beta' \leq \beta$.  
Since $G$ is an amoeba, and  $n\geq n_0(G)$, we know there is a chain  $F = G_0, G_1, G_2, \cdots, G_t = H$ such that, for every $i\in\{1,\cdots, t\}$,  $G_i \cong G$ and $G_i$ is obtained from $G_{i-1}$ by an edge-replacement. Let $r_i=|R\cap E(G_i)|$ be the number of red edges in $G_i$, and  $b_i=|B\cap E(G_i)|$ be the number of blue edges in $G_i$, so that,  for every $i\in\{1,\cdots, t\}$,  $G_i$ is an $(r_i,b_i)$-colored copy of $G$. Observe that an edge-replacement modifies the color pattern in at most one unit, that is, for every $i\in\{1,\cdots, t\}$, $|r_i-r_{i-1}|\leq 1$ as well as $|b_i-b_{i-1}|\leq 1$ and $r_i+b_i=e(G)$. Thus, if we start with an $(r_0,b_0)$-colored copy of  $G$, and we end with an $(r_t,b_t)$-colored copy of $G$, we must cover all $(r,b)$-color patterns with $\alpha = r_0\leq r\leq  r_t = \alpha'$ and $\beta' = b_t\leq b\leq b_0 = \beta$.
\end{proof}

\begin{remark}\label{rem:bip}
Since, by the K\H{o}vari-S\'os-Tur\'an theorem \emph{\cite{KST}}, ${\rm ex}(n,G)=o(n^2)$ for any bipartite graph $G$, we have, for large enough $n$, $2({\rm ex}(n,G)+1)\leq \binom{n}{2}$. This means that we can consider $2$-colorings $E(K_n) = E(R) \cup E(B)$ with $\min \{e(R), e(B)\} \geq {\rm ex}(n,G)+1$ if $n$ is sufficiently large.
\end{remark}

Note that Lemma \ref{lem:basic} implies that, for a given amoeba $G$ and a given $2$-coloring $E(K_n) = E(R) \cup E(B)$, where  $n\geq n_0(G)$, if we can find both a $(0,e(G))$-colored copy and an $(e(G),0)$-colored copy of $G$, then  the so colored $K_n$ will contain the graph $G$ in every possible $(r,b)$-color pattern for $r$ and $b$ with $r+b=e(G)$, $0 \leq r\leq e(G)$ and $0 \leq b\leq e(G)$. Therefore, by means of Lemma~\ref{lem:basic} and Remark \ref{rem:bip}, we can prove our next theorem.

\begin{theorem}\label{thm:amoebasOT}
Every bipartite amoeba $G$ is omnitonal with ${\rm ot}(n,G)={\rm ex}(n,G)$ and ${\rm Ot}(n,G)={\rm Ex}(n,G)$, provided $n$ is large enough to fulfill $\binom{n}{2} \geq 2{\rm ex}(n,G)+1$ and  $n\geq n_0(G)$.
\end{theorem}

\begin{proof}
Let $G$ be  a bipartite amoeba. By Remark \ref{rem:bip} we can consider, for sufficiently large $n$, $2$-colorings of $E(K_n)$ with  $n\geq n_0(G)$ and at least $ {\rm ex}(n,G)+1$ edges of each color. Since, any coloring $E(K_n) = E(R) \cup E(B)$ with $\min \{e(R), e(B)\} \geq {\rm ex}(n,G)+1$ contains a $(0,e(G))$-colored copy of $G$ and an $(e(G),0)$-colored copy of $G$, by Lemma \ref{lem:basic}, there is  an $(r,b)$-colored copy of $G$ for all integers $r$ and $b$ such that $0\leq r, b\leq e(G)$ and $r+b=e(G)$. Thus, $G$ is omnitonal and  ${\rm ot}(n,G)\leq {\rm ex}(n,G)$. In order to see that ${\rm ex}(n,G)\leq {\rm ot}(n,G)$, notice that we can give a $2$-coloring of $E(K_n)$ with $\min \{e(R), e(B)\}= {\rm ex}(n,G)$ such that there are no $(e(G),0)$-colored copies of $G$, and therefore $G$ cannot be omnitonal. Further, observe that the fact that ${\rm ot}(n,G)={\rm ex}(n,G)$ implies that ${\rm Ex}(n,G) \subseteq {\rm Ot}(n,G)$. Suppose now there is a graph $H \in {\rm Ot}(n,G) \setminus {\rm Ex}(n,G)$ and let $E(K_n) = E(R) \cup E(B)$ be a coloring of the edges of $K_n$ such that $R \cong H$. Then $e(R) = {\rm ex}(n,G)$ but, since $R \notin Ex(n,G)$, $R$ contains a subgraph isomorphic to $G$, that is, there is an $(e(G),0)$-copy of $G$ contained in the colored $K_n$. Since  $2{\rm ex}(n,G)+1\leq \binom{n}{2}$, clearly $e(B) \ge {\rm ex}(n,G)+1$ and there is also a $(0,e(G))$-copy of $G$ in $K_n$. Hence, by Lemma \ref{lem:basic}, there is an $(r,b)$-copy of $G$ for every pair of non-negative integers $r,b$ with $r+b = e(G)$, a contradiction to the hypothesis that $R \cong H \in {\rm Ot}(n,G)$. Therefore,  ${\rm Ot}(n,G)={\rm Ex}(n,G)$.
\end{proof}

Since the balanceable property is not as restrictive as the omnitonal property, we will see that we can prescind from the bipartite condition to prove that every amoeba is balanceable. For the proof, we will make use of an old argument of Erd\H{o}s which states that every graph $G$ has a bipartition $V(G)= X \cup Y$ such that $e(X,Y)\geq \left\lceil e(G)/2\right\rceil$ (see Lemma 2.14 in \cite{FuSi}). Deleting edges if necessary, one can easily see that every graph $G$ contains a bipartite subgraph $B$ with $e(B)=\left\lceil e(G)/2\right\rceil$.

\begin{theorem}\label{teo:amoebaBal}
Every amoeba is balanceable.\footnote{Observe that the proof of this theorem yields actually that every amoeba is strongly balanceable in the sence discussed in the footnote of Definition \ref{def:bal}.}
\end{theorem}

\begin{proof}
Let $G$ be an amoeba. By the observation above, we may consider a bipartite subgraph $B$ of $G$ having exactly $e(B)=\left\lceil e(G)/2\right\rceil$ edges. Let $E(K_n) = E(R) \cup E(B)$ be a $2$-coloring  with $\min \{e(R), e(B)\} \geq {\rm ex}(n,B)+1$, which is possible for $n$ large enough because of Remark \ref{rem:bip}. Hence, we know that $K_n$  contains a $(0,e(B))$-colored copy of $B$ and a $(e(B),0)$-colored copy of $B$. Now we can complete those copies of $B$ into copies of $G$ in an arbitrary way to get an $(\alpha,\beta)$-colored copy of $G$, and an $(\alpha',\beta')$-colored copy of $G$, where $\left\lceil e(G)/2\right\rceil\leq \beta$ and $\left\lceil e(G)/2\right\rceil\leq \alpha'$. Since  $\alpha+\beta= \alpha'+\beta' = e(G)$, we also have $\alpha\leq \left\lfloor e(G)/2\right\rfloor$ and $\beta'\leq \left\lfloor e(G)/2\right\rfloor$. Altogether we have $\alpha\leq \left\lfloor e(G)/2\right\rfloor \leq \left\lceil e(G)/2\right\rceil\leq \alpha'$ and $\beta'\leq \left\lfloor e(G)/2\right\rfloor \leq \left\lceil e(G)/2\right\rceil\leq \beta$. Hence, Lemma \ref{lem:basic} implies that $K_n$  contains a $\left(\left\lfloor e(G)/2\right\rfloor, \left\lceil e(G)/2\right\rceil\right)$-copy and a $\left(\left\lceil e(G)/2\right\rceil, \left\lfloor e(G)/2\right\rfloor\right)$-copy of $G$. 
\end{proof}

\begin{remark}
Observe that not every amoeba is bipartite. For example, one can easily check that odd cyles together with a pendant vertex are non-bipartite amoebas. Moreover, there are also omnitonal graphs which are not amoebas. For instance, due to the fact that every tree is omnitonal (Theorem \ref{thm:treeOT}), stars $K_{1,k}$ with $k \ge 3$ leaves are omnitonal, too, but it is evident that they are not amoebas.
\end{remark}

Amoebas are interesting not only because of their good behavior concerning balanceable and omnitonal graphs, but we think they are interesting for their own. A forthcoming paper under preparation \cite{CHM3} will deal with such an analysis.


\section{Balanceable graphs}\label{sec:BAL}

The study of balanceable graphs (in disguise) has already been started in the following three recent papers. The first one is a paper by Caro and Yuster \cite{CaYu}, where zero-sum weighting over $\Z$ are introduced and several zero-sum theorems are proved that fit to the framework of balanceable graphs as explained above. The other two \cite{CHM1, CHM2} develop further the study on $\{-1 ,1\}$-weightings on the set of positive integers $\{1,2,...,n\}$ or on the set of edges of $K_n$, forcing zero-sum copies of given structures (blocks of consecutive integers in the first case,  copies of graphs in the second case) which can be translated into the language of colorings and balanceable graphs.     
 
We restate here, in the language of red-blue coloring, instead of $\{-1 ,1\}$-weighting and zero-sum language, a part of the main theorem from \cite{CHM2}, which is a sort of role-model for the results in this section.

\begin{theorem}[\cite{CHM2}]\label{thm:Km}
\mbox{}
\begin{itemize}
\item[(i)] For any positive integer $m \ge 2$,  $m \neq 4$, $m \equiv 0,1 \;({\rm mod}\; 4)$  the complete graph $K_m$ is not balanceable. 
\item[(ii)] The complete graph $K_4$ is balanceable with ${\rm bal}(n,K_4)= n$, if $n \equiv 0 \;({\rm mod}\; 4)$, and ${\rm bal}(n,K_4)=  n-1$, else. Moreover, ${\rm Bal}(n,K_4)=\{H\}$ with $H=J \cup \bigcup_{i=1}^{q} C_4$, where $J \in \{\emptyset, K_1, K_2, P_2\}$, depending on the residue of $n$ $({\rm mod}\; 4)$, and $q = \lfloor \frac{n}{4} \rfloor$.

\end{itemize}
\end{theorem}

So Theorem \ref{thm:Km} determines which complete graphs with an even number of edges are balanceable and which are not. To show that $K_m$ is not balanceable for $m \equiv 0,1 \;({\rm mod}\; 4)$, we exhibit infinitely many values of $n$ for which there is a balanced red-blue coloring of $E(K_n)$ without a balanced copy of $K_m$ \cite{CHM2}.\\

We know, from Theorem \ref{thm:treeOT}, that  trees are omnitonal and, therefore,  balanceable. In this section, we determine ${\rm bal}(n,G)$ and  describe ${\rm Bal}(n,G)$ for the cases where $G$ is a star or a path with an even number of edges. However, for the general case that $G$ is a tree, the best upper bound we get emerges as a corollary from the bound we obtain for $\rm{ ot}(n,G)$ (see Section \ref{subs:OTtrees}).


\subsection{Stars}

In this section we determine  ${\rm bal}(n,K_{1,k})$ and  describe ${\rm Bal}(n,K_{1,k})$ for $k\geq 2$ even, and $n$ sufficiently large.

\begin{theorem}\label{thm:BP-stars}
Let $k$ and $n$ be integers with $k \ge 2$ even and such that $n\geq \max\{3,\frac{k^2}{4}+1\}$. Then
\[{\rm bal}(n,K_{1,k}) = \left(\frac{k-2}{2}\right)n-\frac{k^2}{8}+\frac{k}{4},
\]
and ${\rm Bal}(n,K_{1,k})$ contains only one graph, namely  the complete $(\frac{k-2}{2},n-\frac{k-2}{2})$-split graph.
\end{theorem}

\begin{proof}
Let $$h(n,k):= \left(\frac{k-2}{2}\right)n-\frac{k^2}{8}+\frac{k}{4}.$$ 
First observe that the condition $\min\{e(R), e(B)\} > h(n,k)$ is satisfiable, that is, we need to prove that $e(K_n) = \frac{n(n-1)}{2} \ge 2 h(n,k) + 2$  holds true for all $n \ge \max\{3,\frac{k^2}{4}+1\}$. If $k =2$, then $h(n,k)=0$ and the condition is satisfied for every $n\geq 3=\max\{3,\frac{k^2}{4}+1\}$. If $k \ge 4$, we have to verify that $2 h(n,k) + 2 = n(k-2)-\frac{k^2}{4}+\frac{k}{2}+2 \le \frac{n(n-1)}{2}$. Equivalently,  $n^2-(2k-3)n+\frac{k^2}{2}-k-4 \ge 0$, which is indeed the case for $n \ge \frac{k^2}{4}+1$ and $k \ge 4$.

Let $H$ be the complete $(\frac{k-2}{2},n-\frac{k-2}{2})$-split graph. We first show that $H$ has exactly $h(n,k)$ edges:
\begin{align}
e(H)&=\frac{1}{2}\left(\frac{k-2}{2}\right)\left(\frac{k-2}{2}-1\right)+\left(\frac{k-2}{2}\right)\left(n-\frac{k-2}{2}\right)\nonumber\\
 &= \left(\frac{k-2}{2}\right)n+\left(\frac{k-2}{2}\right)\left(\frac{k-2}{4}-\frac{1}{2}-\frac{k-2}{2}\right)\nonumber\\
 &= \left(\frac{k-2}{2}\right)n+\left(\frac{k}{2}-1\right)\left(\frac{k}{4}-\frac{k}{2}\right)\nonumber\\
 &= \left(\frac{k-2}{2}\right)n-\frac{k^2}{8}+\frac{k}{4}.\nonumber
 \end{align}

Now, observe that any $2$-coloring $E(K_n) = E(R) \cup E(B)$ where $R$ or $B$ is isomorphic to  $H$ contains no balanced copy of $K_{1,k}$. To see this note that, for such a coloring, there are two types of vertices $v\in V(K_n)$, the ones for which $\{deg_{R}(v), deg_{B}(r)\}=\{0,n-1\}$, and the ones for which $\{deg_{R}(v),deg_{B}(v)\}=\{\frac{k}{2}-1,n-\frac{k}{2}\}$. In any case, it is imposible to have a balanced $K_{1,k}$.

So far, we have proved that ${\rm bal}(n,K_{1,k}) \geq h(n,k)$ and that $H\in {\rm Bal}(n,K_{1,k})$. To prove that ${\rm bal}(n,K_{1,k}) \leq h(n,k)$ and that ${\rm Bal}(n,K_{1,k})=\{H\}$ we will show that any coloring  $E(K_n) = E(R) \cup E(B)$ with $\min \{e(R), e(B)\} \geq h(n,k)$ and such that $R$ and $B$ are not isomorphic to $H$ contains a balanced copy of $K_{1,k}$. For this purpose, we define the following sets
\[V_R = \left\{v \in V(K_n) \;|\; \deg_R(v) \ge \frac{k}{2}\right\}, \mbox{ and} \]\[V_B =\left\{v \in V(K_n) \;|\; \deg_B(v) \ge \frac{k}{2}\right\}.\]
Let $E(K_n) = E(R) \cup E(B)$ be a coloring with $\min \{e(R), e(B)\} \geq h(n,k)$  and such that $R$ and $B$ are not isomorphic to $H$. If there is a vertex $v\in V_{R}\cap V_B$ then we are done as there would be a balanced  $K_{1,k}$. So we may assume that $V_{R}\cap V_B=\emptyset$. Note that,  since every vertex in $K_n$ has degree $n-1\geq k$ then $V(K_n)=V_{R}\cup V_B$, hence
\begin{equation}\label{eq:n}
|V_{R}|+|V_B|=n.
\end{equation}
Assume  without lost of generality that $|V_{R}|\leq |V_{B}|$. \\

\noindent
\emph{Case 1:} Suppose $|V_{R}|\leq \frac{k}{2}-1$. Thus, 
\begin{align}\label{eq:c1}
2e(R)=\sum_{v\in V(K_n)} deg_{R}(v)&=\sum_{v\in V_{R}} deg_{R}(v) \,+ \sum_{v\in V_{B}} deg_{R}(x)\nonumber\\
&\leq |V_R| (n-1) + |V_B|\left(\frac{k-2}{2} \right)\nonumber\\
&= |V_R| (n-1) + (n-|V_R|)\left(\frac{k-2}{2} \right)\nonumber\\
 &\leq \left(\frac{k}{2}-1\right)(n-1)+\left(n-\frac{k}{2}+1\right)\left(\frac{k-2}{2}\right)\nonumber\\
 &= 2\left(\frac{k-2}{2}\right)n-\frac{k^2}{4}+\frac{k}{2}=2h(n,k).
 \end{align}
Consequently, $e(R)\leq h(n,k)$. By assumption we know that $\min \{e(R), e(B)\} \geq h(n,k)$, so we must  have $e(R)= h(n,k)$. Looking back to the inequalities  in (\ref{eq:c1}), it  must be that $|V_R|=\left(\frac{k}{2}-1\right)$ and $R$ is isomorphic to $H$, a contradiction to our assumption.  \\

\noindent
\emph{Case 2:} Suppose know that $|V_R|\geq\frac{k}{2}$. 
Denote by $e'(R)$ the number of red edges  between $V_{R}$ and $V_B$. Since a vertex $v\in V_R$ satisfies $deg_{B}(v)<\frac{k}{2}$ then  each vertex in $V_R$ contributes to $e'(R)$ with at least $|V_{B}|-\frac{k}{2}+1$ edges, thus  
\begin{equation}\label{eq:e1lb}
e'(R)\geq |V_R|\left(|V_{B}|-\frac{k}{2}+1\right) \geq \frac{k}{2}\left(|V_{B}|-\frac{k}{2}+1\right).
\end{equation} 
On the other hand, each vertex in $V_{B}$ contributes to $e'(R)$ with no more than $\frac{k}{2}$ edges, so that
\begin{equation}\label{eq:e1ub}
e'(R)\leq \left(\frac{k}{2}-1\right)|V_{B}|.
\end{equation}
Now, from (\ref{eq:e1lb}) and (\ref{eq:e1ub}), we obtain $$\frac{k}{2}\left(|V_{B}|-\frac{k}{2}+1\right)\leq  \left(\frac{k}{2}-1\right)|V_{B}|,$$  
from which, by means of (\ref{eq:n}) and the assumption that $|V_R|\geq\frac{k}{2}$, it follows that
\begin{equation}\label{eq:v1v-1}
-\frac{k^2}{4}+\frac{k}{2} \leq -|V_{B}|=|V_{R}|-n\leq \frac{k}{2}-n.
\end{equation}
This yields $n\leq \frac{k^2}{4}$, a contradiction to the hypothesis.
\end{proof}

\subsection{Paths}

In this section we determine  ${\rm bal}(n,P_{k})$ and  describe ${\rm Bal}(n,P_{k})$ for $k\geq 2$ even and $n$ sufficiently large.

\begin{theorem}\label{thm:BP-paths}
Let $k \ge 2$ and $n$ be integers with $k$ even and such that $n \geq \frac{9}{32}k^2+\frac{1}{4}k+1$. Then 
\begin{align*}
{\rm bal}(n,P_k) = \left \{\begin{array}{ll}
  \left(\frac{k-2}{4}\right)n-\frac{k^2}{32}+\frac{1}{8}, & \mbox{for }  k\equiv 2  \mbox{ (mod $4$),}\\
\left(\frac{k-4}{4}\right)n-\frac{k^2}{32}+\frac{k}{8}+1, & \mbox{for } k\equiv 0  \mbox{ (mod $4$),} 
\end{array}\right.
\end{align*}
and ${\rm Bal}(n,P_k)$ contains only one graph, namely  the complete  $\left(\frac{k-2}{4},n-\frac{k-2}{4}\right)$-split graph, if $k\equiv 2$  (mod $4$), and the complete $(\frac{k-4}{4},n-\frac{k-4}{4})$-split graph plus one edge, if $k\equiv 0$ (mod $4$).
\end{theorem}

\begin{proof}
Let
\begin{align*}
h(n,k) := \left \{\begin{array}{ll}
  \left(\frac{k-2}{4}\right)n-\frac{k^2}{32}+\frac{1}{8}, & \mbox{for }  k\equiv 2  \mbox{ (mod $4$),}\\
\left(\frac{k-4}{4}\right)n-\frac{k^2}{32}+\frac{k}{8}+1, & \mbox{for } k\equiv 0  \mbox{ (mod $4$),} 
\end{array}\right.
\end{align*}

First observe that the condition $\min\{e(R), e(B)\} > h(n,k)$ is satisfiable, that is, we need to prove that $e(K_n) = \frac{n(n-1)}{2} \ge 2 h(n,k) + 2$  holds true for all  $n \geq \frac{9}{32}k^2+\frac{1}{4}k+1$.
If $k =2$, then $h(n,k)=0$ and the condition is satisfied for every $n\geq 3$. Since $h(n,k)\leq \left(\frac{k-2}{4}\right)n-\frac{k^2}{32}+\frac{k}{8}+1$ we have to verify, for $k \ge 4$, that $ \left(\frac{k-2}{2}\right)n-\frac{k^2}{16}+\frac{k}{4}+4\leq \frac{n(n-1)}{2}.$
Equivalently,  $n^2-n(k-1)+\frac{k^2}{8}-\frac{k}{2}-8 \ge 0$, which is indeed the case for $n \geq \frac{9}{32}k^2+\frac{1}{4}k+1$ and $k \ge 4$.

Let $H$ be the complete  $(\left\lfloor\frac{k-2}{4}\right\rfloor,n-\left\lfloor\frac{k-2}{4}\right\rfloor)$-split graph, plus one edge if $k\equiv 0$ (mod $4$).
 We first show that $H$ has exactly $h(n,k)$ edges. If  $k\equiv 2$ (mod $4$), we get

\begin{align*}
 e(H) &=\frac{k-2}{4}\left(n-\frac{k-2}{4}\right)+\frac{1}{2}\left(\frac{k-2}{4}\right)\left(\frac{k-2}{4}-1\right)\\
 &= \left(\frac{k-2}{4}\right)n -\frac{(k-2)^2}{16}+\frac{1}{2}\frac{(k-2)^2}{16}-\frac{(k-2)}{8}\\
&=\left(\frac{k-2}{4}\right)n -\frac{(k-2)^2}{32}-\frac{(k-2)}{8}\\
&=  \left(\frac{k-2}{4}\right)n-\frac{k^2}{32}+\frac{1}{8}=h(n,k).
\end{align*}

On the other hand, we obtain, for  $k\equiv 0$ (mod $4$),
\begin{align*}
 e(H) &=\frac{k-4}{4}\left(n-\frac{k-4}{4}\right)+\frac{1}{2}\left(\frac{k-4}{4}\right)\left(\frac{k-4}{4}-1\right)+1\\
&= \left(\frac{k-4}{4}\right)n -\frac{(k-4)^2}{16}+\frac{1}{2}\frac{(k-4)^2}{16}-\frac{(k-4)}{8}+1\\
&=\left(\frac{k-4}{4}\right)n -\frac{(k-4)^2}{32}-\frac{(k-4)}{8}+1\\
&=  \left(\frac{k-4}{4}\right)n-\frac{k^2}{32}+\frac{k}{8}+1=h(n,k).
\end{align*}

Now, we show that any $2$-coloring $E(K_n) = E(R) \cup E(B)$ with $\min \{e(R), e(B)\} = h(n,k)$ where $R$ or $B$ is isomorphic to  $H$ contains no balanced copy of $P_k$. Suppose without lost of generality that $R$ is the one isomorphic to $H$. Let $V(K_n)=V_1\cup V_2$ be a partition such that all edges induced by $V_2$, minus one if $k\equiv 0$ (mod $4$), are blue and all remaining edges are red.  A balanced copy of $P_k$ must contains $\frac{k}{2}$ edges of each color;  if $k\equiv 2$ (mod $4$) then $|V_1|=\frac{k-2}{4}$, hence, the maximal number of red edges that a path can contain is $2\left(\frac{k-2}{4}\right)=\frac{k}{2}-1$; if $k\equiv 0$ (mod $4$) then $|V_1|=\frac{k-4}{4}$ and we have and extra red edge in $V_2$, hence the maximal number of red edges that a path can contain is $2\left(\frac{k-4}{4}\right)+1=\frac{k}{2}-1$ Thus, no coloring where $R$ or $B$ is isomorphic to  $H$  can have a balanced $P_k$.

So far, we have proved that ${\rm bal}(n,P_k) \geq h(n,k)$ and that $H\in {\rm Bal}(n,P_k)$. To prove that ${\rm bal}(n,P_k) \leq h(n,k)$ and that ${\rm Bal}(n,P_k)=\{H\}$, we will show by induction on $k$ that any coloring  $E(K_n) = E(R) \cup E(B)$ with $\min \{e(R), e(B)\} \geq h(n,k)$ and such that $R$ and $B$ are not isomorphic to $H$ contains a balanced copy of $P_{k}$.

If $k=2$, then $h(n,2)=0$ and $H$ is the complete $(0,n)$-split graph. It is evident that every $2$-coloring of $E(K_n)$, where $n\geq3$, with at least one edge of each color contains a balanced $P_2$. 

Let $k \ge 4$ and assume that the theorem is valid for $k-2$.  Let $n \geq \frac{9}{32}k^2+\frac{1}{4}k+1$, and consider  $E(K_n) = E(R) \cup E(B)$ with $\min\{e(R) ,e(B)\}\geq h(n,k)$ and such that $R$ and $B$ are not isomorphic to $H$. Since $h(n,k-2)<h(n,k)$ for every $k \ge 4$  and every $n \ge 1$ then, by the induction hypothesis, there exists a  balanced $(k-2)$-path,  say $P = x_1x_2 \ldots x_{k-1}$. Let $U = V(K_n) \setminus V(P)$. Since $n \geq \frac{9}{32}k^2+\frac{1}{4}k+1$, we have $|U|  \ge 2$. Suppose for contradiction that there is no balanced $k$-path. Next, we will analyze the edges from $U$ to $ \{x_1,x_{k-1}\}$ and the edges induced by vertices in $U$.\\

\noindent
 \textit{Claim 1. The edges in $(U, \{x_1,x_{k-1}\})$ are all of the same color.} \\
Suppose there are two edges $x_1u$, $x_{k-1}v$ with $u, v \in U$, $u \neq v$, such that one is  blue and one is red. Then $uPv$ is a balanced $k$-path, a contradiction. $\diamond$\\

In the following, we will assume, without lost of generality, that all edges from $U$ to $ \{x_1,x_{k-1}\}$ are blue. Then,\\

\noindent
 \textit{Claim 2. All edges in E(U) are blue.}\\
Suppose there is a red $uv \in E(U)$. Then, due to Claim 1, $uvP$ is a balanced $k$-path, a contradiction. $\diamond$\\

Let $X = \{x_2,...,x_{k-2}\}$. We will analyze the edges contained in $E(U,X)$. From  Claims 1 and 2 we know that all red edges  in $E(K_n)$ are incident to a vertex in $X$. Let $W\subseteq X$ be the set of vertices that are incident to at least one red edge from  $E(U,X)$. We will call a vertex  $x_i\in  X$ a \emph{red vertex}, if both $x_{i-1}x_i$ and $x_ix_{i+1}$ are red edges. \\

\noindent
 \textit{Claim 3. All vertices in $W$ are red vertices.}\\
Suppose to the contrary that there is a vertex $x_i\in W$ such that $x_{i-1}x_i$ is a blue edge. Since $x_i\in W$, there is a red edge $x_iu$ for some $u\in U$.  Now take $v\in U\setminus \{u\}$ and note that  $vx_{k-1}x_{k-2}\ldots x_iux_1x_2 \ldots x_{i-1}$ is a balanced $k$-path, a contradiction. Hence, $x_{i-1}x_i$ is a  red edge. By a symmetric argument, we can conclude that $x_ix_{i+1}$ must be also red. $\diamond$\\

\noindent
\textit{Claim 4. If $x_i\in W$, then then all edges in $E(U, \{x_{i-1},x_{i+1}\})$ are  blue.}\\
Suppose to the contrary that  $x_{i-1}u$ is a red edge for some $u\in U$. By Claim 3 we know that $x_{i-1}x_i$ is also a red edge. If $ux_i$ is red, take $v\in U$ such that $v\neq u$ and note that  $vx_1\ldots  x_{i-1}ux_{i}\ldots  x_{k-1}$ is a balanced $k$-path, a contradiction. If  $ux_i$ is not red then, since $x_i\in W$, we know that there is a vertex $v\in U$, $v\neq u$, such that $vx_i$ is a red edge.  In this case note that  $x_1\ldots  x_{i-1}uvx_{i+1}\ldots  x_{k-1}$ is a balanced $k$-path, a contradiction. Hence, all edges from $x_{i-1}$ to $U$ are blue edges. By a symmetric argument, we can conclude that all edges from $x_{i+1}$ to $U$ must be also blue. $\diamond$\\

\noindent
 \textit{Claim 5. $|W| =\left\lfloor\frac{k-2}{4}\right\rfloor$ and all red edges of $P$, with exception of one if $k\equiv 0$ (mod $4$), are incident with a vertex in $W$.}\\
By Claim 3, we know that $W$ contains only  red vertices, and, by  Claim 4, we conclude that $W$ contains no consecutive  red vertices. Hence, since $P$ contains exactly $\frac{k-2}{2}$ red edges, $|W|\leq \left\lfloor\frac{k-2}{4}\right\rfloor$. Now suppose for contradiction that $|W| \le \left\lfloor\frac{k-2}{4}\right\rfloor - 1$. Then, considering that all edges from $W$ to $U$ and, with exception of the $\frac{k-2}{2}$ blue edges on $P$, all edges induced by $V(P)$ may be  red edges, we have at most the following number of red edges:
\begin{align*}
e(R) &\leq \left(\left\lfloor\frac{k-2}{4}\right\rfloor-1\right)\left(n-(k-1)\right)+\frac{(k-1)(k-2)}{2} - \frac{k-2}{2}\\
& = \left\lfloor\frac{k-2}{4}\right\rfloor n-(k-1)\left\lfloor\frac{k-2}{4}\right\rfloor -n +k-1+\frac{(k-2)^2}{2}.
\end{align*}

On the other hand, we know by hypothesis that
 \begin{equation} \left\lfloor\frac{k-2}{4}\right\rfloor n-\frac{k^2}{32} \leq h(n,k)\leq e(R).
\end{equation}
Thus, 
$$\left\lfloor\frac{k-2}{4}\right\rfloor n-\frac{k^2}{32} 
\leq e(R) \leq \left\lfloor\frac{k-2}{4}\right\rfloor n-(k-1)\left\lfloor\frac{k-2}{4}\right\rfloor -n +k-1+\frac{(k-2)^2}{2}.$$ 

This gives, together with the inequalities $\lfloor \frac{k-2}{4}\rfloor \ge \frac{k-4}{4}$ that 
\begin{align*}
n &\leq \frac{k^2}{32} -(k-1)\left\lfloor\frac{k-2}{4}\right\rfloor  +k-1+\frac{(k-2)^2}{2}\\
 &\leq \frac{k^2}{32}  - \frac{(k-1)(k-4)}{4} +k-1 +\frac{(k-2)^2}{2}\\
 &= \frac{k^2}{32}  - \frac{k^2-5k+4}{4} +k-1+ \frac{k^2-4k+4}{2}\\
&=\frac{9}{32}k^2 + \frac{1}{4} k,
\end{align*}
a contradiction to the assumption that $n \geq \frac{9}{32}k^2+\frac{1}{4}k+1$.

Note that, to achieve $|W| = \lfloor\frac{k-2}{4} \rfloor$, all red edges from $P$, with exception of one if $k\equiv 0$ (mod $4$), appear in pairs surrounding a vertex from $W$. Therefore, if  $k\equiv 2$ (mod $4$) then all red edges in $P$ are incident with a vertex in $W$ and, if  $k\equiv 0$ (mod $4$) then all red edges except one are incident with a vertex in $W$. Our purpose now is to prove that  the remaining red edges induced by $V(P)$ are all incident with a vertex in $W$. \\

\noindent
 \textit{Claim 6. $x_1x_{k-1}$ is a blue edge.}\\ 
 Suppose that $x_{k-1}x_1$ is red. Then take a blue edge $x_ix_{i+1}$ in $P$ and two vertices $u,v\in U$, $u\neq v$. By Claim 3, we know that $ux_i$ and $vx_{i+1}$ are blue edges, then  $vx_{i+1}\ldots x_{k-1}x_{1}\ldots x_{i}u$ is a balanced $k$-path, a contradiction. $\diamond$\\

\noindent
\textit{Claim 7. If $x_ix_j$ is a red edge for some $1\leq i<j\leq k-1$, $j\neq i+1$, then either $x_i$ or $x_j$ is in $W$.}\\ 
Suppose for contradiction that neither $x_i$ nor  $x_j$ belong to $W$. We will prove the existence of a balanced $k$-path. Consider  $P'=ux_{i+1} \ldots x_jx_i \ldots x_1,x_{k-1} \ldots x_{j+1}v.$ Observe that $$E(P')=\left(E(P)\setminus\{x_ix_{i+1},x_jx_{j+1}\}\right) \cup \{ux_{i+1},vx_{j+1},x_1x_{k-1},x_ix_j\},$$ where $x_1x_{k-1}$ is a blue edge, and $x_ix_j$ is a red edge. Thus, in order to show that $P'$ is a balanced $k$-path, it remains to see that $x_ix_{i+1}$ and $ux_{i+1}$ are edges of the same color as well as $x_jx_{j+1}$ and $vx_{j+1}$. If $x_ix_{i+1}$ and  $x_jx_{j+1}$ are blue, then $ux_{i+1}$ and $vx_{j+1}$ are also blue (by  Claim 3) so we are done. 
Suppose then, without lost of generality,  that $x_ix_{i+1}$ is red. Since $x_i\not \in W$, $x_ix_{i-1}$ must be  blue, which implies that $x_jx_{j-1}$ is red (otherwise, by symmetric arguments we obtain  that the $k$-path $P''=vx_{j-1} \ldots x_ix_j \ldots x_{k-1},x_{1} \ldots x_{i-1}u$ is balanced). Now notice that since $x_j\not \in W$, $x_jx_{j+1}$ must be  blue and the cardinality of $W$ forces that one of $x_{i+1}$ or $x_{j-1}$ belongs to $W$. If $x_{i+1}\in W$, we can choose $u$ such that  $ux_{i+1}$  is a red edge, and we are done; if  $x_{j-1}\in W$, we use the path $P''$ instead of $P'$ to find the balanced $k$-path. In all cases there is a balanced $k$-path wich is a contradiction, and so either $x_i$ or $x_j$ is in $W$. $\diamond$\\

To conclude the proof,  we will count which is the maximum number of possible red edges in $K_n$. Since all edges induced by $W$, and all edges from a vertex in $W$ to a vertex in $V(K_n)\setminus W$, plus one if $k\equiv 0$ (mod $4$),  are the only ones being possibly red, we obtain 

\begin{equation}\label{eq:eR}
e(R) \leq \left\lfloor\frac{k-2}{4}\right\rfloor\left(n-\left\lfloor\frac{k-2}{4}\right\rfloor\right)+\frac{1}{2}\left\lfloor\frac{k-2}{4}\right\rfloor\left(\left\lfloor\frac{k-2}{4}\right\rfloor-1\right)+\epsilon,
\end{equation}
where $\epsilon=0$ if $k\equiv 2$ (mod $4$), and $\epsilon=1$ if $k\equiv 0$ (mod $4$). Note that the right hand of (\ref{eq:eR}) is exactly the number of edges of $H$, that is exactly $h(n,k)$ as shown  at the beginning  of the proof. Since we assume $\min\{e(R) ,e(B)\}\geq h(n,k)$, then $e(R)=h(n,k)$. Moreover, note that $R$ is forced to be isomorphic to $H$, which is a contradiction. So, a balanced $k$-path exist.
\end{proof}


\section{Omnitonal trees}\label{sec:OMNI}

We already know that every omnitonal graph $G$ satisfies ${\rm ot}(n,G) = \mathcal{O}(n^{2-\frac{1}{m}})$, where $m = m(G)$  depends only on $G$ (Corollary \ref{cor:subq}).  We also have seen two classes of graphs which are omnitonal: trees (Theorem \ref{thm:treeOT}) and bipartite amoebas (Theorem \ref{thm:amoebasOT}). Moreover, when a graph $G$ is a bipartite amoeba, Theorem~\ref{thm:amoebasOT} yields  ${\rm ot}(n, G) = {\rm ex}(n, G)$ and ${\rm Ot}(n,G)  = {\rm Ex}(n,G)$. Hence, in particular, we have
\begin{equation}\label{eq:otPk}
{\rm ot}(n, P_k)={\rm ex}(n,P_k)\leq \left( \frac{k-1}{2}\right)n,
\end{equation} where the second inequality is well known \cite{ErGa}). In this  section, we will determine ${\rm ot}(n,G)$ and ${\rm Ot}(n,G)$ for the case that $G$ is a star. We also provide a linear (on $n$) upper bound for ${\rm ot}(n,T)$  where $T$ is a tree. This bound yields naturally an upper bound for ${\rm bal}(n,T)$.

\subsection{Stars}

The following theorem determines ${\rm ot}(n,G)$ and ${\rm Ot}(n,G)$ when $G$ is a star $K_{1,k}$.

\begin{theorem}\label{thm:FS-stars}
Let $n$ and $k$ be positive integers such that $n \ge 4k$. Then
\begin{equation}\label{eq:otSk}
{\rm ot}(n,K_{1,k}) = \left \{\begin{array}{ll}
\left\lfloor \left(\frac{k-1}{2}\right)n\right\rfloor , & \mbox{for }  k \le 3,\\
(k-2)n - \frac{k^2}{2} + \frac{3}{2}k - 1, & \mbox{for } k \ge 4,
\end{array}\right.
\end{equation}
and ${\rm Ot}(n,K_{1,k})$ is the family of graphs containing
\begin{enumerate}
\item the empty graph $\overline{K_n}$, if $k=1$;
\item a disjoint union of $\frac{n}{2}$ $K_2$'s, when $n$ is even, and of $\frac{n-1}{2}$ $K_2$'s and a $K_1$, when $n$ is odd, if $k =2$;
\item a disjoint union of cycles, if $k = 3$;
\item a complete $(k-2,n-k+2)$-split graph, if $k \ge 4$.
\end{enumerate}
\end{theorem}

\begin{proof}
Let \begin{align*}
h(n,k) = \left \{\begin{array}{ll}
\left\lfloor (\frac{k-1}{2})n\right\rfloor, & \mbox{for }  k \le 3,\\
(k-2)n - \frac{k^2}{2} + \frac{3}{2}k - 1, & \mbox{for } k \ge 4,
\end{array}\right.
\end{align*}
First observe that the condition $\min\{e(R), e(B)\} > h(n,k)$ is satisfiable, that is, we need to prove that $e(K_n) = \frac{n(n-1)}{2} \ge 2 h(n,k) + 2$ is satisfied for $n \ge 4k$. If $k \le 3$, it is easy to check that $2h(n,k) + 2 \le n(k-1) + 2 \le \frac{n(n-1)}{2}$. If $k \ge 4$, we have to verify that $2 h(n,k) + 2 = 2n(k-2) - k^2 + 3k \le \frac{n(n-1)}{2}$, which is equivalent to $n^2 -(4k-7)n +2k^2-6k \ge 0$. This is indeed the case for $n \ge 4k$, as we have
\begin{align*}
n^2 -(4k-7)n +2k^2-6k&= \left( n - \frac{4k-7}{2}\right)^2 - \left(\frac{4k-7}{2}\right)^2 + 2k^2 - 6k\\
& \ge \left( 4k - \frac{4k-7}{2}\right)^2 - \left(2k-\frac{7}{2}\right)^2 + 2k^2 - 6k\\
& = \left(2k + \frac{7}{2} \right)^2 - \left(2k - \frac{7}{2} \right)^2+ 2k^2 - 6k\\
& = 28k+ 2k^2 - 6k = 22k+2k^2 \ge 0.
\end{align*}
Next, observe that the colorings described in items 1--3 contains no $(r,b)$-colored copy of $K_{1,k}$ for some pair $(r,b)\in\{(0,k), (k,0)\}$ (that is, there is no blue or red copy of $K_{1,k}$). The coloring of item 4 does not contain a $K_{1,k}$ with $k-1$ blue edges and one red edge or the other way around.  

Now let $E(K_n) = E(R) \cup E(B)$ be a $2$-coloring with $\min \{e(R), e(B)\} \ge h(n,k)$ and such that $R$ and $B$ are not as in items 1--4 from the theorem. We will show that  the so colored $K_n$ contains an $(r, b)$-colored copy of $K_{1,k}$ for every pair $r,b \ge 0$ with $r+b = k$. With this purpose, we define the sets 
\[R_r = \{v \in V(K_n) \;|\; \deg_R(v) \ge r\}, \mbox{ and} \]\[B_b =\{v \in V(K_n) \;|\; \deg_B(v) \ge b\},\]
for integers $b, r \ge 0$ such that $b+r = k$.
If there is a vertex $x \in B_b \cap R_r$ for a pair $b, r \ge 0$ with $b+r = k$, then $x$ is the center of a star $K_{1,k}$ with $b$ blue edges and $r$ red edges. Hence, if $B_b \cap R_r \neq \emptyset$ for every pair $b, r \ge 0$ with $b+r = k$, then $K_n$ contains an $(r, b)$-colored copy of $K_{1,k}$ for every pair $r,b \ge 0$ with $r+b = k$ and we are done. So we may assume that there is a particular pair $b, r \ge 0$ with $b+r = k$ such that $B_b \cap R_r = \emptyset$. Clearly, $V(K_n) \setminus (B_b \cup R_r) = \emptyset$, otherwise there would be a vertex of degree at most $b+r-2=k-2$, which is not possible since every vertex in $K_n$ has degree $n-1 \ge k-1$. Hence, $V(K_n) = B_b \cup R_r$, where the union is disjoint. Observe that $B_0 = R_0 = V(K_n)$ and so, if $b = 0$ and $r = k$, we obtain $R_k = \emptyset$. \\

\noindent
\emph{Case 1: Let $k = 1$.}
Then say $b = 0$ and $r=1$, giving $B_0 = V(K_n)$ and $R_1 = \emptyset$, and thus $R$ is the empty graph, which is not possible by assumption. \\

\noindent
\emph{Case 2: Let $k = 2$.}
Then $\{r,b\} = \{0,2\}$ or $r= b = 1$. Say, in the first case, that $b= 0$ and $r=2$. Then $B_0 = V(K_n)$ and $R_2 = \emptyset$. Then we have
\begin{align*}
2 e(R) \le \left\{
\begin{array}{ll}
n-1, & \mbox{if } n \mbox{ odd},\\
n, & \mbox{if } n \mbox{ even},
\end{array}
\right\} = 2 h(n,2).
\end{align*}
Since by assumption $e(R) \ge h(n,k)$, we obtain equality in the above inequality chain. This is only possible if $R$ is a disjoint union of $\frac{n}{2}$ $K_2$'s, when $n$ is even, and of $\frac{n-1}{2}$ $K_2$'s and a $K_1$, when $n$ is odd, which is not allowed by hypothesis. Hence, $b =r = 1$. Since $\min\{e(R), e(B)\} \ge h(n,k) = \left\lfloor\frac{n}{2}\right\rfloor$, $B_1, R_1 \neq \emptyset$. Thus, $\deg_B(v) = 0$ for all $v \in R_1$, implying that the edges between $B_1$ and $R_1$ are all red. On the other hand, we have also $\deg_R(v) = 0$ for all $v \in B_1$, implying that the edges between $B_1$ and $R_1$ are all blue, a contradiction.\\

\noindent
\emph{Case 3: Let $k = 3$.}
Then $\{r,b\} = \{0,3\}$ or $\{r, b\} = \{1,2\}$. Say, in the first case, that $b= 0$ and $r=3$. Then $B_0 = V(K_n)$ and $R_3 = \emptyset$, and $2 e(R) \le 2n  = 2 h(n,2)$. Since by assumption $e(R) \ge h(n,k)$, it follows that $e(R) = n$ and that all vertices have degree $2$ in $R$, that is, $R$ is a union of cycles, which is not possible by assumption. Thus $\{r, b\} = \{1,2\}$, so say that $b = 1$ and $r=2$. Then $\deg_B(v) = 0$ for all $v \in R_2$, implying that all edges between $B_1$ and $R_2$ are red. But $\deg_R(u) \le 1$ for all $u \in B_1$ and so we infer that $|B_1| \le 1$, which leads us to conclude that there are no blue edges, contradicting the hypothesis $e(B) \ge h(n,3) = n > 0$.\\

\noindent
\emph{Case 4: Let $k \ge 4$.} 
Observe that $B_0 = R_0 = V(K_n)$ and so, if $b = 0$ and $r = k$, we obtain $R_k = \emptyset$, leading to the contradiction $2e(R) \le n(k-1) < 2h(n,k)$. The case $b=k$ and $r=0$ is analogous. Hence, we have $1 \le r, b \le k-1$. If $B_b = \emptyset$, then we would have the same contradiction with
\begin{equation}\label{eq:extremal_case_stars_k<4}
2 e(B) \le n(b-1) \le n(k-1) < 2 h(n,k).
\end{equation}
The same happens if $R_r = \emptyset$. Hence, $B_b , R_r \neq \emptyset$ and, assuming without loss of generality that $|B_b| \le |R_r|$, we have $1 \le |B_b| \le |R_r| \le n-1$. Now we distinguish two cases.\\

\noindent
\emph{Subcase 4.1: Suppose that $|B_b| \le b-1$.} Then we have
\begin{equation}\label{eq:extremal_case_stars_k>3_a}
\begin{array}{ll}
2 e(B) \hspace{-1ex} &= \sum_{v \in R_r} \deg_B(v) + \sum_{v \in B_b} \deg_B(v) \\
& \le (n - |B_b|) (b-1) + |B_b|(n-1)\\
& = |B_b| (n-b) + n(b-1)\\
& \le (b-1)(n-b)+n(b-1) \\
&= -b^2 + (2n+1)b-2n.
\end{array}
\end{equation}

Define the function $g(b) = -b^2 + (2n+1)b-2n$ and observe that $g'(b) = -2b+2n+1 >0$ for $b \in [1,k-1]$. Hence, the maximum of the function $g(b)$ on the domain $[1,k-1]$ is attained when $b = k-1$, and thus
\begin{equation}\label{eq:extremal_case_stars_k>3_b}
\begin{array}{ll}
2 e(B) \hspace{-1ex} &\le -b^2 + (2n+1)b-2n \\
&\le -(k-1)^2 + (2n+1)(k-1)-2n \\
& = 2nk-4n-k^2+3k-2 = 2h(n,k).
\end{array}
\end{equation}
Since, by assumption, $e(B) \le h(n,k)$, we obtain equality all along the inequality chains (\ref{eq:extremal_case_stars_k>3_a}) and (\ref{eq:extremal_case_stars_k>3_b}). This gives us that $b = k-1$, $r=1$, $|B_b| = |B_{k-1}| = b-1 = k-2$, and that each $u \in R_1$ and $v \in B_{k-1}$ have $\deg_B(u) = b-1 = k-2$ and $\deg_B(v) = n-1$. Hence, $B$ is a complete $(k-2,n-k+2)$-split graph, a contradiction to our assumptions.\\

\noindent
\emph{Subcase 4.2: Suppose that $|B_b| \ge b$.}
Considering $e_R(R_r,B_b)$, the number of red edges with one vertex in $R_r$ and one in $B_b$, we have
\begin{eqnarray}\label{ineq:red_edges}
|B_b| (|B_b|-b+1) \le |R_r|(|B_b|-b+1) \le e_R(R_r,B_b) \le |B_b|(r-1).
\end{eqnarray}
In particular, it follows that $|B_b|-b+1 \le r-1$ which is the same as $|B_b| \le r+b-2 = k-2$. Hence, we have $1 \le b \le |B_b| \le k-2$. Moreover, counting the blue edges and using that $deg_B(v) \le b-1$ for every $v \in R_r$ and (\ref{ineq:red_edges}), we obtain
\begin{align*}
2 e(B) &= \sum_{v \in R_r} \deg_B(v) + \sum_{v \in B_b} \deg_B(v) \\
& \le |R_r| (b-1) + |B_b|(n-1) - e_R(R_r,B_b)\\
& \le |R_r| (b-1) + |B_b|(n-1) - |R_r|(|B_b|-b+1)\\
& = |R_r| (-|B_b|+2b-2) + |B_b|(n-1)\\
& = (n - |B_b|) (-|B_b|+2b-2) + |B_b|(n-1)\\
& = |B_b|^2 -2|B_b|b+|B_b| +2nb- 2n.
\end{align*}
Let $g(x,y)$ be the function $g(x,y) = x^2-2xy+x+2ny-2n$. Since $1 \le b \le |B_b| \le k-2$ and by the above inequality chain, $2 e(B) \le g(|B_b|,b)$, we are interested in finding where is the maximum of $g(x,y)$ among the domain $[1,k-2]\times[1,k-2]$ and with the constraint $y \le x$. Observe that the derivatives $\frac{dg}{dx}(x,y) = 2x-2y+1$ and $\frac{dg}{dy}(x,y) = -2x+2n$ are both positive for $(x,y) \in [1,k-2] \times [1,k-2]$ and $y \le x$. So $g(x,y)$ grows in $[1,k-2]\times[1,k-2]$ with $x$ and $y$,  assuming $y \le x$, and the maximum of the function is attained when $x = k-2$ and $y = k-2$. Continuing the computation above, it follows that 
\begin{align*}
2 e(B) & \le |B_b|^2 -2|B_b|b+|B_b| +2nb- 2n = g(|B_b|,b)\\
& \le (k-2)^2 -2(k-2)^2+(k-2)+2n(k-2)-2n\\
& = (k-2)(-(k-2)+1+2n)-2n\\
&= (k-2) (-k+3+2n)-2n\\
&= 2nk-6n-k^2+5k-6 \\
&= 2 \left(n(k-3)-\frac{k^2}{2}+\frac{5}{2}k - 3\right) < 2h(n,k),
\end{align*}
which is a contradiction.\\

Hence, we have shown that $B_b \cap R_r \neq \emptyset$ for every pair $b, r \ge 0$ with $b+r = k$, implying that $K_n$ contains an $(r, b)$-colored copy of $K_{1,k}$ for every pair $r,b \ge 0$ with $r+b = k$. Altogether we have shown that ${\rm ot}(n,K_{1,k}) = h(n,k)$ and that ${\rm Ot}(n,K_{1,k})$ is the family of graphs described in items 1--4.
\end{proof}

\subsection{Trees}\label{subs:OTtrees}
By Theorem \ref{thm:treeOT}, we know that trees are omnitonal. Also, Theorem~\ref{thm:amoebasOT} yields that a tree $T$ which is  an amoeba satisfies ${\rm ot}(n,T) = {\rm ex}(n,T)$, but we also know that not every tree is an amoeba (like stars with at least three leaves). However, we will prove that ${\rm ot}(n,T)$ is linear on $n$ for every tree $t$. More precisely, we will show that more than $(k-1)n$ edges from each color are enough to guarantee the existence of every tree on $k$ edges in all different tonal variations.

In 1962, Erd\H{o}s and S\'os conjecture that the trivial lower bound ${\rm ex}(n,T)\geq n\left(\frac{k-1}{2}\right)$ is tight (see \cite{FuSi}). A proof of this conjecture  for sufficiently large $k$ was announced years ago by Ajtai,  Koml\'os,  Simonovits and  Szemer\'edi, but the proof remains unpublished. For our purpose,  we will use the following weaker statement, which is folklore (see for example \cite{FuSi}). Denote by $\mathcal{T}_k$ the class of trees on $k$ edges. 

\begin{remark}\label{rem}
Let $k$ be a positive integer and let $T\in \mathcal{T}_k$. Then, ${\rm ex}(T,n)<(k-1)n$.
\end{remark}
 
Before stating the theorem, we need one more definition.  In a tree that is not a star, there are at least two vertices such that all but one of its neighbors are leaves. The star induced by such a vertex together with its neighbor-leaves is called an \emph{end-star}. Hence, for every tree $T$ different from a star, there is an end-star vertex $v$ with $\deg(v) \leq \frac{e(T)+1}{2}$.   
 
 \begin{theorem}\label{thm:otTrees}
Let $n$ and $k$ be positive integers such that $n \ge 4k$. Then, for every $T\in \mathcal{T}_k$, ${\rm ot}(n,T) \le (k-1)n$.
\end{theorem}

\begin{proof}
First observe that the condition $\min\{e(R), e(B)\} \geq (k-1)n$ is satisfiable. This is clearly so, since $e(K_n) = \frac{n(n-1)}{2} \ge 2(k-1)n$ is satisfied for $n \geq  4k$. 

We proceed now by induction on $k$. A tree $T$ with $k=1,2,$ or $3$ edges is either a star or a path, and it follows from (\ref{eq:otPk}) and (\ref{eq:otSk}) that  ${\rm ot}(n,T)\leq \left(\frac{k-1}{2}\right)n \leq (k-1)n$, so that the statement holds true in these cases.  For a fix $k\geq 3$, assume that the statement is true for every tree with less than $k$ edges. Let  $T$ be a tree with $k$ edges, let  $n \geq 4k$, and consider a $2$-coloring $E(K_n) = E(R) \cup E(B)$ such that 
\begin{equation}\label{eq:BR}
\min\{e(B),e(R)\} > (k-1)n.
\end{equation} 
Note that,  by Remark \ref{rem}, we get both a $(0,k)$-colored copy, and a $(k,0)$-colored copy of $T$. Thus, we only need to prove that, for every $1\leq r\leq k-1$, there is a $(r,k-r)$-colored copy of $T$ under the  the given coloring. Also note that, if $T$ is a star we are done by Theorem \ref{thm:FS-stars}. Then we assume that $T$ is not a star. By the discussion above the statement of the theorem, there is an end-star vertex $v\in V(T)$ such that $\deg_T(v)=t$ and 
\begin{equation}\label{eq:t}
2\leq t\leq \frac{k+1}{2}.
\end{equation}
Denote by $w$ the only neighbor of $v$ which is not a leaf, and let $T'$ be the tree on $k-t$ edges obtained by deleting $v$ and all its leaf-neighbors from $T$. 
By the induction hypothesis, there is a copy of $T'$ in every tone. For each $0 \le s \le k-t$, let $T'_s$ be an $(s,k-t-s)$-copy of $T'$. Let $W_s = V(K_n) \setminus V(T'_s)$ and let $B^*_s$ and $R^*_s$ be the graphs induced by the blue and, respectively, red edges in $W_s$. Observe that $|W_s| = n - (k-t+1) \ge 4k - (k-t+1) = 3k-t+1 \ge 5t+1 > 4t$. Moreover,
\begin{align*}
\min\{e(B^*_s),e(R^*_s)\} &> (k-1) n - {n(T'_s) \choose 2} - n(T') (n-n(T'_s))\\
& = (k-1)n - \frac{1}{2} (k-t+1)(k-t) - (k-t+1)(n- (k-t+1))\\
& =  ((k-1) - (k-t+1)) (n - (k-t+1)) + (k-t+1) \left((k-1) - \frac{1}{2}(k-t)\right) \\
& =  (t-2) (n - k+t-1) +  \frac{1}{2}  (k-t+1) (k + t-2).
\end{align*}
Since $k \ge 2t -1$ by (\ref{eq:t}), we get $\frac{1}{2}  (k-t+1) (k + t-2) \ge \frac{1}{2}  t (3t-3) > 0$ for $t \ge 2$, and so we can conlcude that 
\[\min\{e(B^*_s),e(R^*_s)\} > (t-2) (n - k+t-1).\]
Hence, by Theorem \ref{thm:FS-stars}, there is a copy of $K_{1,t}$ in $W_s$  in every tone. In particular, there is a $(1,t-1)$ and a $(t-1,1)$ copy of $K_{1,t}$. Now we will show that, for every $1 \le r \le k-1$, there is an $(r,k-r)$-colored copy of $T$ in $K_n$. To this aim, we distinguish two cases.\\

\noindent
\emph{Case 1: $r \ge t-1$.} Set $s = r-t+1$ and consider $T'_s$, the $(s,k-t-s)$-colored copy of $T'$ in $W_s$. Let $w^*$ be the copy of $w$ in $T'_s$. By the discussion above, there is a $(t-1,1)$-colored copy of $K_{1,t}$ in $W_s$ with, say, central vertex $v^*$ and leaves $x_1, x_2, \ldots, x_t$, and such that $vx_i$ is red for $1 \le i \le t-1$, and $v^*x_t$ is blue. In order to find the desired $(r,k-r)$ copy of $T$, proceed as follows:
\begin{itemize}
\item If $w^*v^*$ is red, take the set of vertices $V(T'_s) \cup \{v^*,x_2,\ldots,x_t\}$ and edge set $E(T'_s)\cup \{w^*v^*,v^*x_2, \ldots,v^*x_t\}$.
\item If $w^*v$ is blue, take the set of vertices $V(T_s)\cup  \{v^*,x_1,\ldots,x_{t-1}\}$ and edge set $E(T'_s)\cup \{w^*v^*,v^*x_1, \ldots,v^*x_{t-1}\}$. 
\end{itemize}
In both cases we obtain a copy of $T$ with $s + t-1 = (r-t+1) + t - 1 = r$ red edges and $k-t-s + 1 = k - t - (r-t+1) + 1  = k-r$ blue edges.\\

\noindent
\emph{Case 2: $r \le t-2$.} Set $s = r-1$ and consider $T'_s$, the $(s,k-t-s)$-colored copy of $T'$ in $W_s$. Let $w^*$ be the copy of $w$ in $T'_s$. In this case, we will use a $(1,t-1)$-colored copy of $K_{1,t}$ contained in $W_s$ with, say, central vertex $v^*$ and leaves $x_1, x_2, \ldots, x_t$, and such that $vx_i$ is blue for $1 \le i \le t-1$, and $v^*x_t$ is red. To find the desired $(r,k-r)$ copy of $T$, we do the following:
\begin{itemize}
\item If $w^*v^*$ is red, take the set of vertices $V(T'_s) \cup \{v^*,x_1,\ldots,x_{t-1}\}$ and edge set $E(T'_s)\cup \{w^*v^*,v^*x_1,\ldots,v^*x_{t-1}\}$.
\item If $w^*v$ is blue, take the set of vertices $V(T_s)\cup  \{v^*,x_2, \ldots, x_t \}$ and edge set $E(T'_s) \cup \{ w^*v^*,v^*x_2, \ldots, v^*x_t \}$.
\end{itemize}
In both cases we obtain a copy of $T$ with $s + 1 = (r-1) + 1 = r$ red edges and $(k-t-s) + t- 1 = (k - t - (r-1)) + t -1  = k-r$ blue edges.
\end{proof}

The  bound from Theorem \ref{thm:otTrees} yields a better bound for ${\rm bal}(n,T)$ than the one we get by means of a direct approach using Theorem \ref{thm:BP-stars}. The problem with this approach is that, when considering balanceable graphs, we can deal only with graphs with an even number of edges, and we lose tightness when the subtrees in which we split our tree in the induction step have odd edge number. A better bound may be obtained if the stronger notion of balaceable graphs is considered, as proposed in Footnote \ref{foot:str_bal}.

\begin{corollary}
Let $n$ and $k$ be positive integers such that $n \ge 4k$. Then, for every $T\in \mathcal{T}_k$, ${\rm bal}(n,T) \le (k-1)n$.
\end{corollary}


\section{Open problems and further research}\label{sec:OP}

In this section, we discuss some of the many problems and variations that one can consider.\\ 

Corollary \ref{cor:char-BP} and Theorem \ref{thm:fsp} provide the necessary and sufficient conditions for a graph $G$ to be balanceable and, respectively, omnitonal. From those structural results we can deduce, for instance, that trees are omnitonal, and therefore also balanceable. To find more families of graphs which are omnitonal and/or balanceable is one of our main interests. Hence, we have to deal with the recognition problem and its complexity.

\begin{problem}
What is the computational complexity of determining if a graph $G$ is omnitonal or balanceable?
\end{problem}

When considering the balanceable problem, we have dealt, up to now, only with graphs with an even number of edges. We know that paths and stars with even edge number are balanceable and we have determined the corresponding parameters and extremal colorings. The odd case for these graphs should not be very different from the even case and our proofs may be adapted easily. In the context of complete graphs, we know that $K_4$ is the only balanceable complete graph with an even number of edges. However, the odd case,  which could be interesting, has still to be settled.

\begin{problem}
Let $K_m$ be the complete graph on $m$ vertices, where $m \equiv 2, 3 \;({\rm mod} \; 4)$. Determine for which $m$'s $K_m$ is balanceable and for which not and determine, for the positive cases, ${\rm bal}(n,K_m)$ and ${\rm Bal}(n,K_m)$.
\end{problem}

Of course, there are many other graph families, like certain tree types (with small diameter or bounded maximum degree, for example), cycles, etc. that we have not discussed yet and which could be interesting to study, too. Also, in case of an odd number of edges, the strongly balanceable notion mentioned in the footnote of Definition \ref{def:bal} can be another direction to follow.\\

The graph family of amoebas was born from the study of balanceable and omnitonal graphs, but it is interesting by its own. Thus, the following problems may be considered.

\begin{problem}
What is the computational complexity of determining if a given graph $G$ is an amoeba?
\end{problem}

In this work, we have fully concentrated on the balanceable and the omnitonal problems, but of course one can consider the problem of whether a graph is $r$-tonal as given in Definiton~\ref{def:rtonal}. Clearly, such a graph has to fulfill the conditions of Theorem \ref{thm:char-r-tonal}. Another possibility to consider is, a bit more restrictive, to ask if, given a positive integer $r$, a graph $G$ appears $(r,e(G)-r)$-colored in every $2$-coloring of the complete graph with enough edges from both colors; or one can even look for the stronger version of whether both the $(r,e(G)-r)$-colored as well as the $(e(G)-r,r)$-colored graph $G$ are unavoidable. Certainly, for these versions, there will be a corresponding theorem to Theorem \ref{thm:char-r-tonal}. Another idea is, given a graph $G$ with a fix edge coloring $f$, check if this graph with this particular pattern, say $(G,f)$, is unavoidable in every $2$-coloring of the complete graph with enough edges from each color. By Theorem \ref{thm:RTZ}, this is the case if and only if $(G,f)$ is contained in any balanced type-$A(t)$ as well as in any balanced type-$B(t)$ coloring of $K_n$ (for $n$ large enough and so chosen, along with $t$, that the coloring is balanced). For example, for $k \ge 3$, consider a $k$-path $P$ colored in such a way that there are three consecutive edges with the pattern red-blue-red. Then $P$ is not contained in any balanced type-$A(t)$ coloring of $K_n$ where the red edges induce a complete graph on $t$ vertices. Hence, not every pattern in a path will be unavoidable. On the other side, because the edges in a star play all the same role, and stars are omnitonal, for every pattern of a star we could try to determine the minimum number edges of each color that guarantees the existence of a star with this precise pattern. 

Another natural direction to consider is to have more than two colors. The notion of balanceable graphs, strongly balanceable graphs and omnitonal graphs can be easily modified to the case where $c \ge 3$ colors are used. A generalization of Theorem \ref{thm:RTZ} to $c \ge 2$ colors was already done in \cite{Alp}. However, when considering omnitonal graphs, it is easy to show that, for $c \ge 3$ colors, they must be disconnected, in sharp contrast to the case of $c =2$. This direction is postponed to another paper under preparation.

\section*{Acknowledgements}

We thank  Alp M\"{u}yesser for a discussion on part of this paper, pointing to us references \cite{CuMo,FoSu}, and to his manuscript \cite{Alp} which also contains a sketch of a similar proof of Theorem \ref{thm:RTZ}  concerning $\epsilon$-balanced colorings.\\

The second author was partially supported by PAPIIT IN111819 and CONACyT project 282280. The third author was partially supported by PAPIIT IN116519 and CONACyT project 282280.


\end{document}